\documentclass[a4paper,12pt,reqno]{amsart}

\usepackage[T1]{fontenc}
\usepackage[utf8]{inputenc}
\usepackage{lmodern}
\usepackage[english]{babel}

\usepackage[%
	left=2.5cm,       
	right=2.5cm,      
	top=3.5cm,        
	bottom=3.5cm,     
	heightrounded,    
	bindingoffset=0mm 
]{geometry}

\usepackage{amsmath}
\usepackage{amssymb}
\usepackage{mathtools}
\usepackage{mathrsfs}
\usepackage[normalem]{ulem}

\usepackage[
colorlinks=true,
urlcolor=teal,
citecolor=magenta,
linkcolor=teal,
breaklinks=true]
{hyperref}

\usepackage[noabbrev,capitalize,nameinlink]{cleveref}
\usepackage{comment}
\usepackage{bookmark}

\usepackage[initials,msc-links,
]{amsrefs}

\usepackage[dvipsnames]{xcolor}
\usepackage{graphicx}
\usepackage{url}

\usepackage{enumitem}

\usepackage{aurical} % \hgt

\numberwithin{equation}{section}

\theoremstyle{plain}
\newtheorem{theorem}{Theorem}[section]
\newtheorem{lemma}[theorem]{Lemma}
\newtheorem{proposition}[theorem]{Proposition}

\theoremstyle{definition}
\newtheorem{definition}[theorem]{Definition}
\newtheorem{example}[theorem]{Example}
\newtheorem{remark}[theorem]{Remark}

\newcommand{\di}{\,\mathrm{d}}

\newcommand{\diver}{{\mathrm{div}}}
\newcommand{\exc}{{\mathbf{e}}}
\newcommand{\G}{\mathbb{G}}
\newcommand{\graph}{\mathrm{gr}}
\newcommand{\haus}{\mathscr{H}}
\newcommand{\hgt}{\text{\large\Fontamici h}}
\newcommand{\ilip}{\mathrm{Lip}_{\W}}
\newcommand{\leb}{\mathscr{L}}

\newcommand{\lip}{\mathrm{Lip}}
\newcommand{\loc}{{\mathrm{loc}}}
\newcommand{\maty}{{{\tilde{\mathsf  B}}}}
\newcommand{\mat}{{\mathsf B}}
\newcommand{\matn}{{\mathcal C}}
\newcommand{\N}{\mathbb{N}}
\newcommand{\pp}{\parallel}
\newcommand{\R}{\mathbb{R}}
\newcommand{\shaus}{\mathscr{S}}
\newcommand{\spann}{\mathrm{span}}
\newcommand{\sphere}{\mathbb{S}}
\newcommand{\V}{{\mathbb{V}}}
\newcommand{\W}{{\mathbb{W}}}

\DeclarePairedDelimiter{\abs}{|}{|}
\DeclarePairedDelimiter{\scalar}{\langle}{\rangle}
\DeclarePairedDelimiter{\set}{\{}{\}}
\DeclareMathOperator{\ad}{ad}
\allowdisplaybreaks
\begin{document}

\title[Lipschitz approximation in plentiful groups]{Lipschitz approximation of almost $\mathbb G$-perimeter minimizing boundaries in plentiful groups}

\author[A.~Pinamonti]{Andrea Pinamonti}
\address[A.~Pinamonti]{Dipartimento di Matematica, Università di Trento, via Sommarive 14, 38123 Povo (TN), Italy}
\email{andrea.pinamonti@unitn.it}

\author[G.~Stefani]{Giorgio Stefani}
\address[G.~Stefani]{Scuola Internazionale Superiore di Studi Avanzati (SISSA), via Bonomea 265, 34136 Trieste (TS), Italy}
\email{gstefani@sissa.it {\normalfont or} giorgio.stefani.math@gmail.com}

\author[S.~Verzellesi]{Simone Verzellesi}
\address[S.~Verzellesi]{Dipartimento di Matematica, Università di Trento, via Sommarive 14, 38123 Povo (TN), Italy}
\email{simone.verzellesi@unitn.it}

\date{\today}



\thanks{\textit{Acknowledgements}. 
The authors thank Daniela Di Donato, Roberto Monti, Francesco Serra Cassano and Davide Vittone for interesting and valuable conversations on the topic of the paper.
The authors are members of the Istituto Nazionale di Alta Matematica (INdAM), Gruppo Nazionale per l'Analisi Matematica, la Probabilità e le loro Applicazioni (GNAMPA).
{\color{red} The first-named author has received funding from TBA.}
The second-named author has received funding from INdAM under the INdAM--GNAMPA 2023 Project \textit{Problemi variazionali per funzionali e operatori non-locali}, codice CUP\_E53\-C22\-001\-930\-001, and from the European Research Council (ERC) under the European Union’s Horizon 2020 research and innovation program (grant agreement No.~945655).
The third-named author has received funding from INdAM under the INdAM--GNAMPA 2023 Project \textit{Equazioni differenziali alle derivate parziali di tipo misto o dipendenti da campi di vettori}, codice CUP\_E53\-C22\-001\-930\-001. 
}

\begin{abstract} 
We prove that the boundary of an almost minimizer of the intrinsic perimeter in a plentiful group can be approximated by intrinsic Lipschitz graphs.
Plentiful groups are Carnot groups of step~$2$ whose center of the Lie algebra is generated by any co-dimension one horizontal subspace.
For example, $H$-type groups not isomorphic to the first Heisenberg group are plentiful.
Our results provide the first extension of the regularity theory of intrinsic minimal surfaces beyond the family of Heisenberg groups.
\end{abstract}

\maketitle

\section{Introduction}

\subsection{Framwork}

A \emph{Carnot group} is a Lie group whose Lie algebra admits a suitable stratification in which the first layer---the so-called \emph{horizontal distribution}---generates all the other layers~\cites{BLU07,S16}. 
Non-commutative Carnot groups, endowed with the \emph{Carnot--Carathéodory distance} naturally induced by the horizontal distribution, are not Riemannian at any scale, hence providing an interesting and reach setup for  Analysis. 

The study of Geometric Measure Theory in Carnot groups started from the pioneering work~\cite{FSS01}, and the regularity of sets that are local minimizers for the \emph{horizontal perimeter}, i.e., the perimeter naturally induced by the horizontal distribution, is one of the most important open problems in the field. 
All regularity results known so far are limited to the \emph{Heisenberg groups} $\mathbb H^n$, $n\ge1$, and assume some additional strong \textit{a priori} regularity and/or some restrictive geometric structure of the minimizer~\cites{CapCM09,CCM10,CHY09,SV14,M15}. 
On the other hand, there are examples of minimal surfaces in the first Heisenberg group~$\mathbb H^1$ that are only Lipschitz continuous in the standard sense~\cites{P06,R09}.

The first step in the celebrated De Giorgi's regularity theory for sets of finite perimeter in $\R^n$ is based on a good approximation of the boundary of minimizing sets~\cites{G84,M12}, namely, the so-called \emph{Lipschitz approximation}.
In the original strategy, the approximation is made by convolution and the estimates strongly rely on a
\emph{monotonicity formula}.
However, the validity of such a formula is an open problem in the sub-Riemannian setting~\cite{DGN10}. 
A more flexible approach has been proposed in~\cite{SS82} by means of Lipschitz graphs.
Although the boundary of sets with finite horizontal perimeter may be quite irregular from an Euclidean point of view~\cite{KS04}, the natural notion of \emph{intrinsic Lipschitz graphs}~\cites{FS16,FSS06} turns out to be effective in the approximation within this framework~\cites{M14,MS17,M15}.

\subsection{Main result}

In the present paper, we provide an extension of the approximation by means of intrinsic Lipschitz graphs in the Heisenberg groups $\mathbb{H}^n$ for $n\geq 2$, achieved in~\cites{M14,MS17}, in a more general class of Carnot groups of step~$2$, that we call \emph{plentiful groups}.

In a nutshell, plentiful groups are characterized by the property that any co-dimension~$1$ subspace of the first layer of their Lie algebra still generates the second layer (see \cref{subsec:plentiful}). 
The class of plentiful groups not only includes the important family of \emph{$H$-type groups}~\cite{K80}, but also other interesting examples (see~\cref{example} below). 

Our main result can be stated as follows, see \cref{sec:preliminaries,sec:ilip} for the notation.
For an even more general result concerning \emph{almost minimizers}, see \cref{res:lipa}.

\begin{theorem}[Intrinsic Lipschitz approximation]\label{mainintro}
Let $\G$ be a plentiful group.
For any $L\in(0,1)$, there exist $\varepsilon,C>0$, depending on $L$ only, with the following property.
If $\nu$ is a horizontal direction and $E\subset\G$ is a minimizer of the $\G$-perimeter in the cylinder $C_{324}$ with intrinsic cylindrical excess $\exc(E,0,324,\nu)\le\varepsilon$ and $0\in\partial E$, then, letting 
\begin{equation*}
M=C_1\cap\partial E,
\quad
M_0=\set*{
q\in M : \sup_{0<r<16}\exc(E,q,r,\nu)\le\varepsilon},
\end{equation*}
there exists an intrinsic Lipschitz funciton $\varphi\colon\W\to\R$ such that 
\begin{equation*}
\sup_\W|\varphi|\le L,
\quad
\ilip(\varphi)\le c_\G\,L,
\end{equation*}
\begin{equation*}
M_0\subset M\cap\Gamma,
\quad
\Gamma=\graph(\varphi;D_1),
\end{equation*}
\begin{equation*}
\shaus_\infty^{Q-1}(M\bigtriangleup\Gamma)
\le 
C\,\exc(E,0,324,\nu),
\end{equation*}
\begin{equation*}
\int_{D_1}|\nabla^\varphi\varphi|^2\di\leb^{n-1}
\le 
C\,\exc(E,0,324,\nu),
\end{equation*}
where $c_\G>0$ is a structural constant independent of $L$.
\end{theorem}

\cref{mainintro} perfectly generalizes~\cite{M14}*{Th.~5.1} to plentiful groups, in fact providing a sub-optimal version of the Lipschitz approximation proved in~\cite{MS17}*{Th.~3.1} in $\mathbb H^n$ for $n\ge2$ (also see~\cite{M12}*{Th.~23.7} for the corresponding result in the Euclidean setting).

In \cref{mainintro}, differently from the corresponding result in~\cite{MS17}, the constants $\varepsilon$ and $C$ may depend on the chosen Lipschitz constant $L$. 
This is due to the current lack of an analog of the deep \emph{height estimate} proved in~\cite{MV15} for $\mathbb H^n$, with $n\ge2$, in plentiful groups. 
However, we believe that the algebraic framework provided by plentiful groups is the correct setting where to possibly extend \cref{mainintro} to its optimal version. The validity of the height estimate in the context of plentiful groups will be the object of future works.

\subsection{Organization of the paper}
The rest of the paper is organized as follows. 
In \cref{sec:preliminaries}, we fix the notation and we recall some basic preliminaries. 
In \cref{subsec:plentiful}, we introduce the class of plentiful groups and we study their main properties. 
In \cref{sec:ilip}, we recall some facts about intrinsic cones, intrinsic Lipschitz graphs and the intrinsic area formula. 
Finally, in \cref{ilasec}, we prove our main result.

\section{Preliminaries}\label{sec:preliminaries}	

We recall the main notation and results used throughout the paper.
For a thorough introduction on the subject, we refer to~\cites{FSS03,BLU07,S16} concerning Carnot groups, and to~\cite{M12} for the usual approach to the regularity theory for minimal surfaces in the Euclidean setting.

\subsection{Carnot groups}

A Carnot group $(\G,\star)$ is a connected, simply connected and nilpotent Lie group whose Lie algebra $\mathfrak{g}$ of left-invariant vector fields has dimension $n$ and admits a stratification of step $s\in\N$, that is, 
\begin{equation*}
\mathfrak{g}=V_1\oplus V_2\oplus\cdots\oplus V_s,
\end{equation*}
where the vector spaces $V_1,\dots,V_s\subset\mathfrak g$ satisfy
\begin{equation*}
V_i=[V_1,V_{i-1}]\quad \text{for}\ i=1,\dots,s-1, \quad [V_1,V_s]=\set{0}. 
\end{equation*}
We set $m_i=\dim(V_i)$ for $i=1,\dots,s$. 
We fix an \emph{adapted basis} of $\mathfrak{g}$, i.e.\ a basis $X_1,\dots,X_n$ such that
\begin{equation*}
X_{h_{i-1}+1},\dots,X_{h_i}\ \text{is a basis of}\ V_i,\quad i=1,\dots,s.
\end{equation*}
We endow the algebra $\mathfrak{g}$ with the left-invariant Riemannian metric $\scalar*{\cdot,\cdot}$ that makes the basis $X_1,\dots,X_n$ orthonormal.
Exploiting the exponential identification $p=\exp\left(\sum_{i=1}^np_iX_i\right)$, we can identify $\G$ with $\R^n$, endowed with the group law determined by the Campbell--Hausdorff formula.
In particular, the identity $e\in\G$ corresponds to $0\in\R^n$ and the \emph{inversion map} becomes $\iota(p)=p^{-1}=-p$ for any $p\in\G$. Moreover, it is not restrictive to assume that $X_i(0)=\mathrm{e}_i$ for any $i=1,\dots,n$. 
Therefore, by left-invariance, we get
\begin{equation*}
X_i(p)=d\tau_p\mathrm{e}_i, \quad i=1,\dots,n,
\end{equation*}
where $\tau_p\colon\G\to\G$ is the \emph{left-translation} by $p\in\G$, i.e.\ $\tau_p(q)=p\star q$ for any $q\in\G$. 

For any $i=1,\dots,n$, the \emph{degree} $d(i)\in\set*{1,\dots,\kappa}$ of the basis vector field $X_i$ is $d(i)=j$ if and only if $X_i\in V_j$. 
The group \emph{dilations} $(\delta_\lambda)_{\lambda\ge0}\colon\G\to\G$ are hence given by
\begin{equation*}
\delta_\lambda(p)=\delta_\lambda(p_1,\dots,p_n)
=(\lambda p_1,\dots,\lambda^{d(i)} p_i,\dots,\lambda^s p_n) 
\quad
\text{for all}\ p\in\G.
\end{equation*}

The Haar measure of the group $\G$ coincides with the $n$-dimensional Lebesgue measure~$\leb^n$.
The homogeneity property $\leb^n(\delta_\lambda(E))=\lambda^Q\leb^n(E)$ holds for any measurable set $E\subset\G$, where $Q=\sum_{i=1}^\kappa i\dim(V_i)\in\N$ is the \emph{homogeneous dimension} of $\G$.
For notational convenience, we use the shorthand $|E|=\leb^n(E)$. 

Following~\cite{FSS03}*{Th.~5.1}, we fix the left-invariant and homogeneous distance $d_\infty(p,q)=d_\infty(q^{-1}\cdot p,0)$ for $p,q\in\G$, where, identifying $\G=\R^n=\R^{m_1}\times\dots\times\R^{m_s}$ as above and letting $\pi_{\R^{m_i}}\colon\R^n\to\R^{m_i}$ be the canonical projection for $i=1,\dots,s$, 
\begin{equation}
\label{eq:distance}
d_\infty(p,0)
=
\max\set*{\epsilon_i|\pi_{\R^{m_i}}(p)|_{\R^{m_i}}^{1/i} : i=1,\dots,s}
\quad
\text{for all}\ p\in\G
,
\end{equation}
with constants $\epsilon_1=1$ and $\epsilon_i\in(0,1)$ for all $i=2,\dots,s$ depending on the structure of $\G$.
We use the shorthand $\|p\|_\infty=d_\infty(p,0)$ for $p\in\G$.
Consequently, for $p\in\G$ and $r>0$, we define the open and closed balls
\begin{equation*}
B_r(p)
=
\set*{q\in\G : d(q,p)<r},
\quad
\bar B_r(p)
=
\set*{q\in\G : d(q,p)\le r},
\end{equation*}
with the shorthands $B_r=B_r(0)$ and $\bar B_r=\bar B_r(0)$.

\subsection{Sets of finite perimeter}

A set $E\subset\G$ is of \emph{locally finite $\G$-perimeter} in an open set $\Omega\subset\G$ if there exists a $\R^{m_1}$-valued Radon measure $\mu_E$ on $\Omega$ such that 
\begin{equation*}
\int_E\diver_\G\phi\di x
=
-
\int_\Omega\scalar{\phi,\di\mu_E}
\quad
\text{for all}\
\phi\in C^1_c(\Omega;\R^{m_1}).
\end{equation*}
Here and in the following, $\diver_\G\phi=\sum_{i=1}^{m_1}X_i\phi_i$ is the  \emph{horizontal divergence} of $\phi$.
If $|\mu_E|(\Omega)<+\infty$, then $E$ has \emph{finite $\G$-perimeter} in $\Omega$.
We also use the notation $P(E;A)=|\mu_E|(A)$ for any Borel $A\subset\G$ and the shorthand $P(E)=P(E;\G)$.
If $E\subset\G$ has (Euclidean) Lipschitz topological boundary $\partial E$, then 
\begin{equation}
\label{eq:perimeter_lip}
P(E;\Omega)
=
\int_{\partial E\cap\Omega}
\left(
\sum_{i=1}^{m_1}\scalar*{N_E,X_i}^2
\right)^{1/2}\di\haus^{n-1},
\end{equation}
where $N_E$ is the standard (inner) unit normal to $\partial E$ and $\haus^s$ is the standard $s$-Hausdorff measure in $\R^n$, $s\in[0,n]$.
By the Radon--Nykodim Theorem, there is a Borel function $\nu_E\colon\Omega\to\R^{m_1}$, called (\emph{measure-theoretic}) \emph{inner horizontal normal} of $E$ in $\Omega$, such that $\mu_E=\nu_E|\mu_E|$ with $|\nu_E|=1$ $\mu_E$-a.e.\ in $\Omega$.
The \emph{reduced boundary} of $E$ is the set $\partial^*E$ of $p\in\G$ such that 
\begin{equation*}
p\in\mathrm{supp}|\mu_E|
\quad
\text{and}
\quad
\nu_E(p)=\lim_{r\to0^+}\frac{\mu_E(B_r(p))}{|\mu_E|(B_r(xp))}\in\sphere^{m_1}.
\end{equation*}
The (\emph{measure-theoretic}) \emph{boundary} of a measurable set $E\subset\G$ is 
\begin{equation}
\label{eq:boundary}
\partial E
=
\set*{p\in \G : |E\cap B_r(p)|>0\
\text{and}\
|E^c\cap B_r(p)|>0\
\text{for all}\ r>0
}.
\end{equation}
Up to modify a set $E\subset\G$ of locally finite $\G$-perimeter in an $\leb^n$-negligible way, arguing \emph{verbatim} as in~\cite{M12}*{Prop.~12.19}, we can always assume that $\partial E$ coincides with the topological boundary of~$E$.

\subsection{Perimeter minimizers}

Let $\Omega\subset\G$ be a (non-empty) open set and let $E\subset\G$  be a set with locally
finite $\G$-perimeter in $\G$. 
We say that the set $E$ is a
\emph{$(\Lambda,r_0)$-minimizer of the $\G$-perimeter in~$\Omega$} if there exist $\Lambda\in[0,+\infty)$ and $r_0\in(0,+\infty]$ such that
\begin{equation*}
P(E;B_r(p))\le P(F;B_r(p))+\Lambda\,|E\bigtriangleup F|
\end{equation*}
for any measurable set $F\subset\G$, $p\in\Omega$ and $r<r_0$ such that
$E\bigtriangleup F\Subset B_r(p)\Subset\Omega$.
If $\Lambda=0$ and $r_0=\infty$, then $E$ is a \emph{locally
$\G$-perimeter minimizer in~$\Omega$}, that is, 
\begin{equation*}
P(E;B_r(p))\le P(F;B_r(p))
\end{equation*}
for any measurable set $F\subset\G$, $p\in\Omega$  and $r>0$ such that
$E\bigtriangleup F\Subset B_r(p)\Subset\Omega$.
 
\begin{remark}[Scaling of $(\Lambda,r_0)$-minimizers]
\label{rem:scaling}
If the set $E$ is a $(\Lambda,r_0)$-minimizer of the $\G$-perimeter in $\Omega\subset\G$, then the set $E_{p,r}=\delta_{\frac{1}{r}}(\tau_{p^{-1}}(E))$ is a $(\Lambda',r_0')$-minimizer of the $\G$-perimeter in $\Omega_{p,r}=\delta_{\frac{1}{r}}(\tau_{p^{-1}}(\Omega))$ for every $p\in\G$ and $r>0$, where
$\Lambda'=\Lambda r$ and $r_0'=r_0/r$. 
In particular, the product $\Lambda r_0$ is invariant by dilation, and thus it is convenient to assume that $\Lambda r_0\le1$, as we
shall always do in the following.  
\end{remark}

\subsection{Carnot groups of step 2}
\label{subsec:step_2}

From now on, we work in a Carnot group $(\G,\star)$ of step $s=2$, so that $\mathfrak g=V_1\oplus V_2$, $[V_1,V_1]=V_2$, $[V_1,V_2]=\set*{0}$, $n=m_1+m_2$ and $Q=m_1+2m_2$.
We fix an adapted orthonormal basis $X_1,\dots,X_{m_1},T_1,\dots,T_{m_2}$ of~$\mathfrak g$, so that $X_1,\dots,X_{m_1}$ and $T_1,\dots ,T_{m_2}$ are orthonormal bases of $V_1$ and $V_2$, respectively.
As well known (see~\cite{BLU07}*{Sec.~3.2} for instance), exploiting exponential coordinates associated to $X_1,\dots,X_{m_1},T_1,\dots,T_{m_2}$,
\begin{equation}
\label{eq:group_law_2}
p\star q
=
(x,t)\star(\xi,\tau)
=
\left(x+\xi,t+\tau+\tfrac12\scalar*{\mat x,\xi}\right)
\end{equation}
for $p,q\in\G$, with $p=(x,t)$, $q=(\xi,\tau)$, $x,\xi\in\R^{m_1}$, $t,\tau\times\R^{m_2}$, where $\mat =(\mat ^1,\dots,\mat ^{m_2})$ is an $m_2$-tuple of linearly independent skew-symmetric $m_1\times m_1$ matrices and 
\begin{equation*}
\scalar*{\mat x,\xi}
=
\left(\scalar*{\mat ^1x,\xi},\dots,\scalar*{\mat ^{m_2}x,\xi}\right)\in\R^{m_2}.
\end{equation*} 
With this notation, we recognize that $\|p\|_\infty=\max\set*{|x|,\epsilon_2\sqrt{|t|}}$ and $\delta_\lambda(p)=\delta_\lambda(x,t)=(\lambda x,\lambda^2 t)$ for $\lambda\ge0$ and $p=(x,t)\in\G$.
Finally, we let $\matn\in(0,+\infty)$ be such that 
\begin{equation}
\label{eq:b_constold}
\abs*{\scalar*{\mat x,\xi}}
\le
\matn\,|x|\,|\xi|
\quad
\text{for all}\
x,\xi\in\R^{m_1}.
\end{equation}

\subsection{Stratified changes of coordinates}\label{coc}

Let $X_1',\dots,X_{m_1}'$ be another orthonormal basis of $V_1$.
Given $p\in\G$, let $p=(x',t)$ be the exponential coordinates associated with the adapted basis $X_1',\dots,X_{m_1}',T_1,\dots,T_{m_2}$. 
Then $x'=Mx$, for a suitable orthogonal $m_1\times m_1$ matrix~$M$. 
Being $M$ orthogonal, $\|\cdot\|_\infty$ is not affected by this change of coordinates. Moreover, in these new coordinates,
\begin{equation*}
    p\star q=(x',t)\star(\xi',\tau)=\left(x'+\xi',+\tfrac12\scalar*{\maty x',\xi'}\right),
\end{equation*}
where $\maty =(\maty ^1,\dots,\maty ^{m_2})$ and $\maty^j=M\mat^jM^T$ for any $j=1,\ldots,m_2$. Notice that
\begin{equation*}
    \sup_{x'\neq 0}\frac{|\maty^j x'|}{|x'|}= \sup_{x'\neq 0}\frac{|M\mat^jM^T x'|}{|x'|}=\sup_{x'\neq 0}\frac{|\mat^jM^T x'|}{|x'|}=\sup_{x'\neq 0}\frac{|\mat^jM^T x'|}{|M^Tx'|}=\sup_{x\neq 0}\frac{|\mat^jx|}{|x|}
\end{equation*}
for any $j=1,\ldots,m_2$, which in turn implies that
\begin{equation}\label{eq:b_const}
    \abs*{\scalar*{\maty x',\xi'}}
\le
\matn\,|x'|\,|\xi'|,
\end{equation}
with the same constant $\matn$ as in \eqref{eq:b_constold}. We stress that, although the above change of coordinates induces an isometry of $\mathfrak g$, it may not be a group morphism (e.g., see~\cite{M03}*{Ex.~2.15}). 
In fact, a simple computation shows that $M$ induces a group morphism if and only if $\mat^jM=M\mat^j$ for every $j=1,\ldots,m_2$.

\subsection{Further properties of perimeter minimizers}

In a Carnot group $\G$ of step~$2$, locally finite $\G$-perimeter sets enjoy further regularity properties, see~\cite{FSS03}*{Sec.~3}.
In particular, for any set $E\subset\G$ with locally finite $\G$-perimeter, 
\begin{equation*}
P(E;A)=\shaus^{Q-1}_\infty(\partial^*E\cap A)
\quad
\text{for each Borel}\
A\subset\G,
\end{equation*}
see~\cite{FSS03}*{Th.~3.10} and~\cite{M17}*{Th.~1.3} (as well as the discussion around~\cite{S16}*{Th.~5.18}).
Here and in the rest of the paper, for any $E\subset\G$ we let
\begin{equation*}
\shaus^s_\infty(E)
=
\sup_{\delta>0}\shaus^s_{\infty,\delta}(E),
\end{equation*}
be the \emph{spherical $s$-Hausdorff measure of $E$} (relative to $d_\infty$ in~\eqref{eq:distance}), where, for any $\delta>0$,
\begin{equation*}
\shaus^s_{\infty,\delta}(E)
=
\inf\set*{
c_\G\sum_{i\in\N}
(\mathrm{diam}_{d_\infty} B_i)^s
:
E\subset\bigcup_{i\in\N} B_i,
\
B_i\ \text{$d_\infty$-ball with}\ \mathrm{diam}_{d_\infty} B_i<\delta 
},
\end{equation*}
where $c_\G>0$ is a renormalizing constant that we do not need to specify here.
We can state the following results concerning the properties of $(\Lambda,r_0)$-minimizers of the $\G$-perimeter in Carnot groups of step~$2$.
The proofs are straightforward adaptations of those for $(\Lambda,r_0)$-minimizers of the Euclidean perimeter in $\R^n$, see~\cite{M12}*{Ch.~21}.
\begin{theorem}[Density estimates]
\label{res:density}
There exist $c_1,c_2,c_3,c_4>0$ such that, if $E\subset\G$ is a $(\Lambda,r_0)$-minimizer of the $\G$-perimeter in the open set $\Omega\subset\G$, $\Lambda r_0\le 1$, $p\in\partial E\cap\Omega$, $B_{r_0}(p)\subset\Omega$, then
\begin{equation*}
c_1\le\frac{|E\cap B_r(p)|}{r^Q}\le c_2
\quad
\text{and}
\quad
c_3\le\frac{\mu_E(B_r(p))}{r^{Q-1}}\le c_4
\quad
\text{for}\ 
r\in(0,r_0).
\end{equation*}
In particular, $\shaus^{Q-1}_\infty\big((\partial E\setminus\partial^*E)\cap\Omega\big)=0$. 
\end{theorem}

\begin{proof}
The result follows by adapting the proof of~\cite{M12}*{Th.~21.11}, invoking~\cite{FSS03}*{Lem.~2.21 and Prop.~2.23}.
Details are omitted.
\end{proof}

\begin{theorem}[Compactness]
\label{res:compactness}
If $(E_j)_{j\in\N}$ is a sequence of $(\Lambda,r_0)$-minimizers of the $\G$-perimeter in the open set $\Omega\subset\G$, $\Lambda r_0\le 1$, then there exist a subsequence $(E_{j_k})_{k\in\N}$ and a $(\Lambda,r_0)$-minimizer of the $\G$-perimeter $E\subset\G$ in $\Omega$ such that 
\begin{equation*}
E_{j_k}\to E
\quad
\text{in}\ L^1_\loc(\Omega)
\quad
\text{and}
%\quad
%\nu_{E_{j_k}}\mu_{E_{j_k}}
%\overset{*}\rightharpoonup 
%\nu_E\,\mu_E
%\quad
\quad
|\mu_{E_{j_k}}|
\overset{*}\rightharpoonup 
|\mu_E|
\quad
\text{as}\ k\to+\infty.
\end{equation*}
Moreover, $(\partial E_{j_k})_{k\in\N}$ converges to $\partial E$ in the sense of Kuratowski, i.e.: 
\begin{enumerate}[label=(\roman*),itemsep=1ex,topsep=1ex]

\item 
if $p_{j_k}\in\partial E_{j_k}\cap\Omega$ and $p_{j_k}\to p\in\Omega$ as $k\to+\infty$, then $p\in\partial E$;

\item
if $p\in\partial E\cap\Omega$, then there exist $p_{j_k}\in\partial E_{j_k}\cap\Omega$ such that $p_{j_k}\to p$ as $k\to+\infty$.

\end{enumerate}
\end{theorem}

\begin{proof}
The result follows by adapting the proof of~\cite{M12}*{Prop.~21.13 and Th.~21.14}, exploiting the density estimates provided by \cref{res:density}.
Details are omitted.
\end{proof}

We underline that \cref{res:density,res:compactness} contain the only properties concerning $(\Lambda,r_0)$-minimizers of the $\G$-perimeter needed in the rest of the paper.

\subsection{Complementary subgroups}\label{complesubsect}
As in~\cite{FS16}*{Sec.~4}, we consider two \emph{complementary} subgroups $\W$ and $\V$ of $\G$, i.e., such that $\V\cap\W=\set*{0}$ and $\G=\W\star\V$.
We also assume that $\V$ is a $1$-dimensional and, consequently, horizontal subgroup of $\G$.
Precisely, we have 
\begin{equation*}
\V=\set*{\exp(sV) : s\in\R}
\quad
\text{for some $V\in V_1$ with $|V_1|=1$.}
\end{equation*}
In the following, in the spirit of \cref{coc}, we will often choose an orthonormal basis $X_1,X_2\ldots,X_{m_1}$ of $V_1$ adapted to the decomposition $\W\star\V$, that is $X_1=V_1$ and
\begin{equation*}
\V=\exp(\spann\set*{X_1}),
\quad 
\W=\exp(\spann\set*{X_2,\dots,X_{m_1},T_1,\dots,T_{m_2}}).
\end{equation*}
We naturally identify 
\begin{equation*}
\V\equiv\R,
\quad
\W\equiv\set*{p=(x,t)\in\G : x_1=0}\equiv\R^{n-1}.
\end{equation*} 
In particular, $w\in\R^{n-1}$ is identified with $(0,w)\in\G$.
Consequently, given $A\subset\W\equiv\R^{n-1}$, any function $\varphi\colon A\subset\W\to\V$ can be identified with a function $\varphi\colon A\subset\R^{n-1}\to\R$.

\subsection{Height function and projections}
For a given $\nu\in \mathbb{R}^{m_1}$ with $|\nu|=1$, we let the group homomorphism 
\begin{equation*}
\hgt\colon\G\to\R,
\quad
\hgt(p)=\langle\nu,x\rangle
\quad
\text{for}\ p=(x,t)\in\G,
\end{equation*}
be the \emph{height function}.
We let $\pi_\V\colon\G\to\V$, $\pi_\V(p)=\hgt(p)\nu$ for $p\in\G$, where, with a slight abuse of notation, we identify $\nu$ with $(\nu,0)\in\mathbb G$, be the  \emph{projection on~$\V$}. 
Moreover, we let $\pi_\W\colon\G\to\W$, uniquely given by the relation 
\begin{equation}
\label{eq:W_proj}
p=\pi_\W(p)\star\pi_\V(p)
\quad
\text{for}\ p\in\G,
\end{equation}
be the \emph{projection on $\W$}.
Using the shorthands $x^\pp=\hgt(p)\nu$ and $x^\perp=x-x^\pp$ for $x\in\R^{m_1}$, 
and exploiting~\eqref{eq:group_law_2} we easily get that, for $p=(x,t)\in\G$,  
\begin{equation*}
\pi_\V(p)
=
(x^\pp,0),
\quad
\pi_\W(p)
=
\left(
x^\perp,
t-\tfrac12\scalar*{\mat  x^\perp,x^\pp}
\right),
\end{equation*}
owing to the fact that $\scalar*{\mat y,y}=0$ for any $y\in\R^{m_1}$ by skew-symmetry.
Let us also observe that,
for $w\in\R^{n-1}$ and $s\in\R$, 
\begin{equation*}
w\star(s\mathrm \nu)
=
\exp(s\nu)(w)
\quad
\text{in}\ \G,
\end{equation*}
where, again with an abuse of notation, we identify $\nu$ with its associated left-invariant vector field. 
Finally, by definition, we can estimate
\begin{equation}
\label{eq:norm_V}
\|\pi_\V(p)\|_\infty
=
|\langle x,\nu\rangle|
\le
|x|
\le
\|p\|_\infty
\end{equation}
and, consequently, 
\begin{equation}
\label{eq:norm_W}
\|\pi_\W(p)\|_\infty
=
\|p\star\pi_\V(p)^{-1}\|_\infty
\le 
\|p\|_\infty+\|\pi_\V(p)^{-1}\|_\infty
=
\|p\|_\infty+\|\pi_\V(p)\|_\infty
\le 
2\|p\|_\infty.
\end{equation}
\subsection{Disks and cylinders}

We let 
\begin{equation*}
D_r=\set*{w\in\W : \|w\|_\infty<r}
\end{equation*}
be the \emph{open disk centered at $0\in\W$ of radius $r>0$}, and we set $D_r(w)=w\star D_r$ for any $w\in\W$.
Note that $\leb^{n-1}(D_r(w))=\leb^{n-1}(D_1)\,r^{Q-1}$ for all $r>0$ and $w\in\W$. 
We also let 
\begin{equation*}
C_r
=
D_r\star(-r,r)
=
\set*{w\star(s\mathrm \nu) : w\in D_r,\ s\in(-r,r)}
\end{equation*}
be the \emph{open cylinder with central section $D_r$ and height $2r$}, and we set $C_r(p)=p\star C_r$ for any $p\in\G$.
We also let
\begin{equation*}
A\star\R
=
\set*{w\star(s\mathrm \nu) : w\in A,\ s\in\R}
\end{equation*}
be the \emph{open infinite cylinder with central section $A\subset\W$}.
In virtue of~\eqref{eq:W_proj}, we have that 
\begin{equation*}
p\in C_r
\iff
\pi_\W(p)\in D_r,\
\hgt(p)\in(-r,r)
\iff
\|\pi_\W(p)\|_\infty<r,\
|\hgt(p)|<r
\end{equation*}
Thanks to the inequalities~\eqref{eq:norm_V} and~\eqref{eq:norm_W}, the left-invariant map $\|\cdot\|_C\colon\G\to[0,+\infty)$,
\begin{equation}
\label{eq:norm_C}
\|p\|_C=\max\set*{\|\pi_\W(p)\|_\infty,|\hgt(p)|}
\quad
\text{for}\
p\in\G,
\end{equation}
is a quasi-norm such that  $C_r=\set*{p\in\G : \|p\|_C<r}$ and
\begin{equation}
\label{eq:norms_inftyC}
\|p\|_C\le2\|p\|_\infty,
\quad
\|p\|_\infty\le2\|p\|_C,
\quad
\text{for}\ p\in\G.
\end{equation}
Consequently, $d_C\colon\G\times\G\to[0,+\infty)$, $d_C(p,q)=\|q^{-1}\star p\|_C$ for $p,q\in\G$, is a left-invariant quasi-distance on~$\G$ and  
\begin{equation}
\label{eq:inclusions_BC}
B_{r/2}(p)\subset C_r(p)\subset B_{2r}(p)
\quad
\text{for all}\
p\in\G,\ r>0.
\end{equation}

\subsection{Cylindrical excess}

A concept which plays a key role in the regularity theory of $(\Lambda,r_0)$-minimizers of the $\G$-perimeter is the \emph{cylindrical excess}, see~\cite{M12}*{Ch.~22} for the Euclidean setting  and~\cites{M14,MV15,M15,MS17} for the Heisenberg groups.

\begin{definition}[Cylindrical excess]
\label{def:excess}
The \emph{cylindrical excess} of a locally finite $\G$-perimeter set  $E\subset\G$ at $p\in\partial E$, at scale $r>0$, and with respect to the horizontal direction $\nu$, is 
\begin{equation*}
\begin{split}
\exc(E,p,r,\nu)
&=
\frac{1}{2\,r^{Q-1}}\int_{C_r(p)}
|\nu_E(p)-\nu|^2\di\mu_E(p)
\\
&=
\frac{1}{r^{Q-1}}\int_{C_r(p)\cap\partial^*E}
\left(1-\scalar*{\nu_E(p),\nu}^2\right)\di\shaus^{Q-1}_\infty(p).
\end{split}
\end{equation*} 
If no confusion arises, we set 
$\exc(p,r)=\exc(E,p,r)=\exc(E,p,r,\nu)$ and $\exc(r)=\exc(0,r)$.
\end{definition}
The basic properties of the cylindrical excess introduced in \cref{def:excess} can be plainly recovered from the corresponding ones known in the Euclidean setting, see~\cite{M12}*{Ch.~22}, and the Heisenberg groups, see~\cite{M14}*{Sec.~3} and~\cite{MV15}*{Sec.~3B}.
We omit the statements.
The following result corresponds to~\cite{MV15}*{Lem.~3.4 and Cor.~3.5}, which were stated in the setting of the Heisenberg groups $\mathbb H^n$, $n\ge2$ (also see~\cite{M12}*{Lem.~22.11} for the Euclidean case).
The very same results hold for any Carnot group of step $2$, with identical proof.

\begin{lemma}[Excess measure]
\label{res:exc_meas}
Let $E\subset\G$ be a set with locally finite $\G$-perimeter with $0\in\partial E$.
If there exists $s_0\in(0,1)$ such that 
\begin{equation*}
\sup\set*{|\hgt(p)| : p\in C_1\cap\partial E}
\le 
s_0,
\end{equation*} 
\begin{equation*}
\leb^{n-1}\left(\set*{p\in E\cap C_1 : \hgt(p)>s_0}\right)=0,
\end{equation*}
\begin{equation*}
\leb^{n-1}\left(\set*{p\in  C_1\setminus E : \hgt(p)<-s_0}\right)=0,
\end{equation*}
then, for a.e.\ $s\in(-1,1)$ and any $\phi\in C_c(D_1)$, letting
\begin{equation*}
M=C_1\cap\partial^*E,
\quad
M_s=M\cap\set*{\hgt>s},
\quad
E_s=\set*{w\in\W : w\star(s\mathrm \nu)\in E},
\end{equation*} 
we have
\begin{equation*}
\int_{E_s\cap D_1}\phi\di\leb^{n-1}
=
\int_{M_s}\phi\circ\pi_\W\,\scalar*{\nu_E,\nu}\di\shaus^{Q-1}_\infty.
\end{equation*}
Consequently, for any Borel set $G\subset D_1$,
\begin{equation*}
\leb^{n-1}(G)
=
\int_{M\cap\pi_\W^{-1}(G)}\scalar*{\nu_E,\nu}\di\shaus^{Q-1}_\infty,
\end{equation*}
\begin{equation}
\label{eq:small_height_est}
\leb^{n-1}(G)
\le 
\shaus^{Q-1}_\infty\left(M\cap\pi_\W^{-1}(G)\right).
\end{equation}
Moreover, we have 
\begin{equation*}
0\le 
\shaus^{Q-1}_\infty(M_s)-\leb^{n-1}(E_s\cap D_1)\le \exc(E,0,1)
\quad
\text{for a.e.}\
s\in(-1,1),
\end{equation*}
\begin{equation*}
\shaus^{Q-1}_\infty(M)-\leb^{n-1}(D_1)=\exc(E,0,1).
\end{equation*}
\end{lemma}
\color{black}

\section{Plentiful groups}
\label{subsec:plentiful}

Contrarily to what happens in $\R^n$, the fact that $\exc(E,p,r)=0$ for some $p\in\partial E$ and $r>0$ does not necessarily imply that $\partial E$ is flat in a neighborhood of~$p$.
This indeed happens in the first Heisenberg group~$\mathbb H^1$, see the example in~\cite{MSV08}*{Th.~1.5} and the characterization provided by~\cite{M14}*{Prop.~3.7}.
Nevertheless, this is not the case for any Heisenberg group $\mathbb H^n$ with $n\ge2$, as proved in~\cite{M14}*{Prop.~3.6}.
Consequently, in order to avoid minimal surfaces with zero excess that are not flat, we need to restrict our attention to a special class of Carnot groups, defined as follows.
%\begin{definition}[Plentiful group]
%\label{def:plentiful}
%We say that a Carnot group~$\G$ of step~$2$ is \emph{plentiful} if there exists $V\subset V_1$ with $\dim V=m_1-1$ such that $[V,V]=V_2$.
%\end{definition}
\begin{definition}[Plentiful group]
\label{def:plentiful}
We say that a Carnot group~$\G$ of step~$2$ is \emph{plentiful} if any $V\subset V_1$ with $\dim V=m_1-1$ satisfies $[V,V]=V_2$.
\end{definition}
The property of being plentiful is well behaved with respect to Lie group isomorphisms.
\begin{proposition}
\label{res:plentiful_isom}
Let $\G_1$ and $\G_2$ be two Carnot groups of step $2$.
If $\G_1$ is plentiful and $\phi\colon\G_1\to\G_2$ is a Lie groups isomorphism, then also $\G_2$ is plentiful. 
\end{proposition}

\begin{proof}
Set $\mathfrak g_1=V_1\oplus V_2$ and $\mathfrak g_2=W_1\oplus W_2$, with $V_2=[V_1,V_1]$ and $W_2=[W_1,W_1]$.
Note that $\mathrm{d}\phi\colon\mathfrak g_1\to\mathfrak g_2$ is an isomorphism preserving the stratification of the corresponding algebras.
Hence, letting $W\subset W_1$ be as in \cref{def:plentiful} for $\G_2$, $V=(d\phi)^{-1}(W)$ is an $(m-1)$-dimensional vector subspace of $V_1$. Thus, since $\G_1$ is plentiful, we get that
\begin{equation*}
[W,W]=[d\phi (V),d\phi (V)]
=
d\phi([V,V])
=
d\phi (V_2)
=
W_2,
\end{equation*}
proving that also $\G_2$ is plentiful.
\end{proof}

We observe that the first Heisenberg group $\mathbb H^1$ is not plentiful. 
More generally, every \emph{free} Carnot group of step $2$ (see~\cite{BLU07}*{Sec.~3.3} for the precise definition) is not plentiful.
On the other hand, the Heisenberg group $\mathbb H^n$ is plentiful for any $n\geq 2$.
More in general, we have the following result.
\begin{theorem}
\label{res:H-type}
An $H$-type group is plentiful if and only if it is not isomorphic to $\mathbb H^1$.
\end{theorem}
We recall that a Carnot group $\mathbb G$ of step $2$ is of \emph{$H$-type} if, for any $Z\in V_2$, the map $J_Z\colon V_1\to V_1$ given by
\begin{equation}
\label{eq:H-type}
\langle J_Z(X),Y\rangle=\langle Z,[X,Y]\rangle
\quad
\text{for any}\ X,Y\in V_1 
\end{equation}
is orthogonal whenever $|Z|=1$.
Notice that $\mathbb H^n$ is of $H$-type for all $n\ge1$.

\begin{proof}[Proof of \cref{res:H-type}]
Let $T_1,\ldots,T_{m_2}$ be an orthonormal basis of $V_2$ and let $X\in V_1$. 
By~\cite{BLU07}*{Prop.~18.1.8}, for any $X\in V_1$, it holds that $X,J_{T_1}(X),\dots,J_{T_{m_2}}(X)$ is an orthonormal subfamily of $V_1$, hence yielding that $m_1\geq m_2+1$. Fix $V\subset V_1$ as in \cref{def:plentiful} and let $v\in V_1\cap V^\perp$ be such that $|v|=1$.
We now distinguish two cases. 

\vspace{1ex}

\textit{Case~1}.
Let us assume that $m_1>m_2+1$. 
In view of~\eqref{eq:H-type} and~\cite{BLU07}*{Prop.~18.1.8}, $J_{T_1}(v),\dots,J_{T_{m_2}}(v)$ is hence an orthonormal subfamily of~$V$. Moreover, again owing to the fact that $m_1>m_2+1$, there exists $w\in V$ which is orthogonal to $J_{T_1}(v),\dots,J_{T_{m_2}}(v)$ and satisfies $|w|=1$. Again by~\cite{BLU07}*{Prop.~18.1.8}, we get
\begin{equation*}
    \left\langle v,J_{T_j}(w)\right\rangle=-\left\langle w,J_{T_j}(v)\right\rangle=0
\end{equation*}
for any $j=1,\ldots,m_2$, which implies that $J_{T_j}(w)\in V$ for any $j=1,\ldots,m_2$. Since $[w,J_{T_j}(w)]=T_j$ for each $j=1,\dots,m_2$ by~\eqref{eq:H-type}, we conclude that $[V,V]=V_2$, as desired.

\vspace{1ex}

\textit{Case~2}.
Now assume that $m_1=m_2+1$. 
We can assume that $m_1>2$, since otherwise $\G$ is isomorphic to $\mathbb H^1$. 
We recall that $\G$ is of $H$-type if and only if, for any $X\in V_1$ with $|X|=1$, the map $\ad_X =[X,\,\cdot\,]$ is a surjective isometry from $\ker(\ad_X)^\perp\cap V_1$ to $V_2$, see~\cites{K80,CDKR91}. 
Since $m_1=m_2+1$, we infer that $\ker(\ad_X)^\perp\cap V_1=X^\perp\cap V_1$. Let $X\in V$ be such that $|X|=1$. 
By the previous considerations, $\dim(\ad_X(V\cap X^\perp))=m_2-1$. Let $T\in V_2\cap \ad_X(V\cap X^\perp)^\perp$ be such that $|T|=1$. Since $[X,J_T(X)]=T$ and $\ad_X$ is injective, we infer that, up to a sign, $v=J_T(X)$. Since $m_1>2,$ and hence $\dim(V)>1$, let $Y\in V$ be such that $|Y|=1$ and $\langle X,Y\rangle=0$. By~\cite{K80}, we infer that 
\begin{equation*}
    \left\langle J_T(Y),v\right\rangle=\left\langle J_T(Y),J_T(X)\right\rangle=-\left\langle Y,J_T^2(X)\right\rangle=\langle Y,X\rangle=0,
\end{equation*}
and so $J_T(Y)\in V$. 
Since $[Y,J_T(Y)]=T$, we get $[V,V]=V_2$, concluding the proof.
\end{proof}

We point out that the class of plentiful groups is broader than that of $H$-type groups.

\begin{example}\label{example}
Consider the stratified Lie algebra $\mathfrak g_{7,5,2}$ of dimension~$7$, rank~$5$ and step~$2$, with only non-trivial commutation relations given by
\begin{equation*}    
[X_1,X_2]=[X_3,X_4]=T_1,\quad[X_1,X_5]=[X_2,X_3]=T_2
\end{equation*}
(for a construction, see~ \cite{LT22}*{(27B)}).
Let $\G_{7,5,2}$ be its associated Carnot group. 
In view of ~\cite{BLU07}*{Prop.~18.1.5}, $\G_{7,5,2}$ is not of $H$-type. 

We claim that $\G_{7,5,2}$ is plentiful. 
To this aim, let us fix $V\subset V_1$ as in \cref{def:plentiful} and let $v\in V_1\cap V^\perp$ be such that $|v|=1$. We let $v=\sum_{j=1}^5a_jX_j$, where $a_j=\langle v,X_j\rangle$. 
We now observe that $W_j=X_j-a_jv\in V$ for $j=1,\dots,5$ are such that 
\begin{equation}\label{sommacom}
[W_1,W_4]+[W_2,W_3]
=
\alpha T_2,
\quad
\text{with}\
\alpha=
\left(a_1^2+\left(a_4^2-a_4a_5+a_5^2\right)\right)\ge0,
\end{equation}
and 
\begin{equation}
\label{eq:treno}
\begin{split}
[W_1,W_2]
&=
\left(1-a_1^2-a_2^2\right)
T_1
+
\left(a_1a_3-a_2a_5\right)
T_2,
\\[1ex]
[W_1,W_5]&
=
-a_2a_5\,T_1
+
\left(1-a_1^2-a_5^2\right)
T_2,
\\[1ex]
[W_3,W_4]
&=
\left(1-a_3^2-a_4^2\right)
T_1
+
a_2a_4\,T_2.
\end{split}
\end{equation}
We now distinguish two cases, depending on whether $\alpha=0$ or $\alpha>0$ in~\eqref{sommacom}.

If $\alpha=0$, then $a_1=a_4=a_5=0$.
Due to~\eqref{eq:treno}, we get $[W_1,W_5]=T_2$, $[W_1,W_2]=a_3^2\,T_1$ and $[W_3,W_4]=a_2^2\,T_1$, proving the claim, since either $a_2\neq 0$ or $a_3\neq 0$.

If $\alpha>0$ instead, then $T_2\in [V,V]$ by~\eqref{sommacom}.
Therefore, by~\eqref{eq:treno}, we get $(1-a_3^2-a_4^2)T_1\in V$ and $(1-a_1^2-a_2^2)T_1\in V$. 
If $a_3^2+a_4^2\ne1$, then $T_1\in V$. 
If $a_3^2+a_4^2=1$ instead, then $a_1=a_2=0$ and so $T_1\in V$, proving the claim.
\end{example}

\color{black}
Our interest for plentiful groups is encoded in the following result, which is a sort of localized version of~\cite{FSS03}*{Lem.~3.6}.
This result is essential in the proof of  \cref{res:flat}, where we prove that plentiful groups do not admit non-flat surfaces with zero excess. 

\begin{lemma}
\label{res:flat_level_sets} 
Let $\G$ be a plentiful Carnot group.
Let $\Omega\subset\G$ be a non-empty connected open set and let $Z_1,\dots,Z_{m_1}$ be an orthonormal basis of $V_1$.  
If $f\in L^1_\loc(\G)$ is such that $Z_1f\ge0$ and $Z_if=0$ for $i=2,\dots,m_1$ in $\Omega$, then the level sets of $f$ in $\Omega$ coincide with left translations of $\set*{p\in\G : \scalar*{p,Z_1(0)}=0}$.
\end{lemma}

\begin{proof}
We can assume $f\in C^\infty(\G)$, since the general case can be recovered by approximation.
Clearly, $Z_iZ_jf=0$ for all $i,j=2,\dots,m_1$ and thus, since $\G$ is plentiful, $Tf=0$ in $\Omega$ for any $T\in V_2$.
Since the left-invariant distribution $\mathscr D$ generated by the vector fields $\mathfrak g\setminus\spann\set*{Z_1}$ is involutive, $\G$ is foliated by smooth $(n-1)$-dimensional manifolds tangent to $\mathscr D$ which, in $\Omega$, coincide with the level sets of $f$.
Since $Z_1,\dots,Z_{m_1}$ are orthonormal and left-invariant, each leaf of the foliation coincides with the leaf passing through $0\in\G$, that is, $\set*{p\in\G : \scalar*{p,Z_1(0)}=0}$, up to a left translation.
\end{proof}

The following crucial result extends~\cite{M14}*{Prop.~3.6} to plentiful groups.
We notice that \cref{res:flat} below can be achieved as~\cite{M14}*{Prop.~3.6} by a straightforward adaptation of~\cite{M14}*{Lem.~3.5}.
However, we prove \cref{res:flat} via a different and plainer argument, somewhat reminiscent of the proof of~\cite{FSS03}*{Claim~3.7}, by exploiting  \cref{res:flat_level_sets}.
\begin{theorem}[Locally constant normal]
\label{res:flat}
Let $\G$ be a plentiful Carnot group. Let $E\subset\G$ be a set with finite $\G$-perimeter in $B_r(p)$, for $p\in\partial E$ and $r>0$.
If $\nu_E(q)=\nu$ for $\mu_E$-a.e.\ $q\in B_r(p)$, then
\begin{equation*}
E\cap B_r(p)
=
\set*{q\in B_r(p) : \hgt(q)>\hgt(p)}
\end{equation*}
up to $\leb^n$-negligible sets.
\end{theorem}

\begin{proof}
We can clearly assume that $p=0$ up to a translation.
Take $\zeta\in\R^{m_1}$ and consider the left-invariant differential operator 
$L_\zeta
=
\sum_{j=1}^{m_1}\zeta_jX_j$
and the test horizontal vector field $\phi=\zeta\psi\in C^1_c(B_r;\R^{m_1})$ for some arbitrary $\psi\in C^1_c(B_r;\R)$.
By assumption, we can compute
\begin{align*}
\int_EL_\zeta\psi\di\leb^n
=
\int_E\diver_\G\phi\di\leb^n
=
-
\int_{B_r}\scalar*{\phi,\nu_E}\di\mu_E
=
-\int_{B_r}\psi\,\scalar*{\zeta,\nu}\di\mu_E,
\end{align*}
yielding that $L_\zeta\mathbf{1}_E=0$ if $\scalar*{\zeta,\nu}=0$ and $L_\zeta\mathbf{1}_E\ge0$ if $\zeta=\nu$ in $B_r$.
By \cref{res:flat_level_sets}, 
\begin{equation*}
E\cap B_r
=
\tau_q\left(\set*{\tilde q\in\G : \hgt(\tilde q)>0}\right)\cap B_r
\quad
\text{for some}\ q\in\G.
\end{equation*}
To conclude, we just need to show that $\hgt(q)=0$, as this yields 
\begin{equation*}
\tau_q\left(\set*{\tilde q\in\G : \hgt(\tilde q)>0}\right)
=
\set*{\tilde q\in\G : \hgt(\tilde q)>0}.
\end{equation*}
Indeed, if $\hgt(q)>0$, then $B_\rho\cap \tau_q\left(\set*{\tilde q\in\G : \hgt(\tilde q)>0}\right)=\emptyset$ for some $\rho\in(0,r)$, yielding
\begin{equation*}
|B_\rho\cap E|
=
\left|B_\rho\cap \tau_q\left(\set*{\tilde q\in\G : \hgt(\tilde q)>0}\right)\right|
=
0,
\end{equation*}
against the assumption that $0\in\partial E$, recall~\eqref{eq:boundary}.
The case $\hgt(q)<0$ can be similarly addressed by considering $E^c$ in place of $E$. 
The proof is complete.
\end{proof}
\color{black}
\section{Intrinsic cones, Lipschitz graphs and area formula}
\label{sec:ilip}

Throughout this section, we assume that $(\G,\star)$ is a Carnot group of step $2$ as in \cref{subsec:step_2}.
For a general introduction about the topics of this section, we refer to~\cite{FS16}. Moreover, here and for the rest of the paper, we fix an horizontal direction $\nu$ and we choose an adapted basis $\nu=X_1,X_2,\ldots,X_{m_1},T_1,\ldots,T_{m_2}$ of $\mathfrak g$ as in \cref{coc,complesubsect}. 
In the induced exponential coordinates, we write $p=(x,t)$ for any $p\in\mathbb G$.\color{black}

\subsection{Intrinsic cones}
\label{subsec:cones}

The following definition rephrases~\cite{M14}*{Def.~4.3} and~\cite{FS16}*{Def.~9}. 

\begin{definition}[Intrinsic cones]
The open \emph{$X_1$-cone} with vertex $0\in\G$ and aperture $\alpha\in(0,+\infty]$ is the set
\begin{equation*}
C(0,\alpha)
=
\set*{p\in\G : \|\pi_\W(p)\|_\infty<\alpha\|\pi_\V(p)\|_\infty}.
\end{equation*} 
The corresponding \emph{negative} and \emph{positive cones} are
\begin{equation*}
C^\pm(0,\alpha)
=
\set*{p=(x,t)\in\G : \|\pi_\W(p)\|_\infty<\alpha\|\pi_\V(p)\|_\infty,\ x_1\gtrless0}.
\end{equation*}
Consequently, we let $C(p,\alpha)=p\star C(0,\alpha)$ and $C^\pm(p,\alpha)=p\star C^\pm(0,\alpha)$ for $p\in\G$.
\end{definition}

Note that, given $p=(x,t)\in\G$ and $\alpha\ge0$, $\|\pi_\W(p)\|_\infty\le\alpha\|\pi_\V(p)\|_\infty$ rewrites as
\begin{equation}
\label{eq:cone_explicit}
\max\set*{ 
\abs*{x^\perp}, \epsilon_2\abs*{t-\tfrac{1}{2}\scalar*{\mat x^\perp,x^\pp}\,}^{1/2}} 
\le 
\alpha
|x_1|.
\end{equation}

The following result collects some elementary properties of cones in Carnot groups of step~$2$, generalizing~\cite{M14}*{Lem.~4.5}.
We briefly detail its proof for the ease of the reader.

\begin{lemma}[Properties of cones]
\label{res:cones}
The following hold:
\begin{enumerate}[label=(\roman*),itemsep=1ex]

\item
\label{item:cone_space} 
$\displaystyle\bigcup_{s<s_0}C^+(p\star s\mathrm e_1,\alpha)=\G
$
for all $\alpha>0$, $p\in\G$ and $s_0\in\R$;

\item
\label{item:cone_inv}
$\displaystyle C^-(0,\alpha)\subset \iota\left(C^+\left(0,\alpha+\epsilon_2\sqrt{\alpha\matn}\right)\right)$
for all $\alpha>0$;

\item
\label{item:cone_abc}
$C^\pm(p,\beta)\subset C^\pm(0,\gamma)$
for all
$p\in C^\pm(0,\alpha)$, with $\alpha,\beta\ge0$ and
\begin{equation*}
\gamma=\max\set*{\alpha,\beta,\tfrac{\epsilon_2}2\sqrt{\left(\alpha\beta+2\beta\right)\matn}},
\end{equation*}
where $\matn>0$ is the constant in~\eqref{eq:b_constold}.
\end{enumerate}
\end{lemma}

\begin{proof}
We prove each statement separately.

\vspace{1ex}

\textit{Proof of~\ref{item:cone_space}}.
Assume $p=0$ and note that, in virtue of~\eqref{eq:b_const} and~\eqref{eq:cone_explicit}, we can compute 
\begin{equation*}
\begin{split}
C^+(s\mathrm e_1,\alpha)
&=
s\mathrm e_1\star
C^+(0,\alpha)
\\
&=
s\mathrm e_1\star
\set*{(x,t)\in\G : 
\max
\set*{ 
\abs*{x^\perp},\epsilon_2\abs*{t-\tfrac12\scalar*{\mat x^\perp,x^\pp}}^{1/2}}< 
\alpha
x_1
}
\\
&=
\set*{(x,t)\in\G : 
\max
\set*{ 
\abs*{x^\perp},\epsilon_2\abs*{t-\tfrac12\scalar*{\mat x^\perp,x^\pp-2s\mathrm e_1}}^{1/2}} 
< 
\alpha
(x_1-s)
}.
\end{split}
\end{equation*}
Hence~\ref{item:cone_space} for $p=0$ follows from the fact that, for any $(x,t)\in\G$, there is $\sigma\in\R$ such that 
\begin{equation*}
\epsilon_2\abs*{t-\tfrac12\scalar*{\mat x^\perp,x^\pp-2s\mathrm e_1}}^{1/2} 
< 
\alpha
(x_1-s)
\quad
\text{for all}\ s<\sigma.
\end{equation*}
By left translation, \ref{item:cone_space} holds for any $p\in\G$.

\vspace{1ex}

\textit{Proof of \ref{item:cone_inv}}.
For any $\beta>0$ we have that 
\begin{equation*}
\iota\left(C^+(0,\beta)\right)
=
\set*{(x,t)\in\G : 
\max
\set*{ 
\abs*{x^\perp},\epsilon_2\abs*{t+\tfrac12\scalar*{\mat x^\perp,x^\pp}}^{1/2}}
<-\beta x_1
}. 
\end{equation*}
Hence, if $(x,t)\in C^-(0,\alpha)$, then
$\abs*{\scalar*{\mat x^\perp,x^\pp}}\le\matn\abs*{x^\perp}\abs*{x^\pp}<\alpha\matn\abs*{x^\pp}^2$ and so 
\begin{equation*}
\begin{split}
\epsilon_2\abs*{t+\tfrac12\scalar*{\mat x^\perp,x^\pp}}^{1/2}
\le 
\epsilon_2\abs*{t-\tfrac12\scalar*{\mat x^\perp,x^\pp}}^{1/2}
+
\epsilon_2\abs*{\scalar*{\mat x^\perp,x^\pp}}^{1/2}
<
-\left(\alpha+\epsilon_2\sqrt{\alpha\matn}\right)x_1,
\end{split}
\end{equation*}
proving~\ref{item:cone_inv}.

\vspace{1ex}

\textit{Proof of \ref{item:cone_abc}}.
If $p=(x,t)\in C^+(0,\alpha)$, then
\begin{equation}
\label{eq:ping}
\max
\set*{\abs*{x^\perp},\epsilon_2\abs*{t-\tfrac12\scalar*{\mat x^\perp,x^\pp}}^{1/2}}
\le 
\alpha x_1.
\end{equation} 
Moreover, if $q\in C^+(p,\beta)$, then $q=p*w$ with $w=(\xi,\tau)\in\G$ such that
\begin{equation}
\label{eq:pong}
\max
\set*{\abs*{\xi^\perp},\epsilon_2\abs*{\tau-\tfrac12\scalar*{\mat \xi^\perp,\xi^\pp}}^{1/2}}
\le 
\beta\xi_1.
\end{equation}
Now, since 
$
q
=
(x,t)\star(\xi,\tau)
=
\left(x+\xi,t+\tau+\tfrac12\scalar*{\mat x,\xi}\right)
$, 
we can write 
\begin{equation*}
\|\pi_\W(q)\|_\infty
=
\max\set*{
\abs*{x^\perp+\xi^\perp},
\epsilon_2\abs*{t+\tau+\tfrac12\scalar*{\mat x,\xi}-\tfrac12\scalar*{\mat (x^\perp+\xi^\perp),x^\pp+\xi^\pp}^{1/2}}
}.
\end{equation*}
Since $\scalar*{\mat x^\pp,\xi^\pp}=0$, by~\eqref{eq:b_const} we easily see that 
\begin{equation}
\label{eq:racchetta}
\begin{split}
\abs*{\scalar*{\mat x,\xi}-\scalar*{\mat x^\perp,\xi^\pp}-\scalar*{\mat \xi^\perp,x^\pp}}
&=
\abs*{\scalar*{\mat x^\perp,\xi^\perp}+2\,\scalar*{\mat x^\pp,\xi^\perp}}
\\
&\le
\matn\left(\abs*{x^\perp}\abs*{\xi^\perp}+2\abs*{x^\pp}\abs*{\xi^\perp}\right)
\\
&\le
\matn\left(\alpha\beta+2\beta\right)
\abs*{x^\pp}\abs*{\xi^\pp}.
\end{split}
\end{equation}
Therefore, by the triangle inequality, \eqref{eq:ping}, \eqref{eq:pong} and~\eqref{eq:racchetta}  yield that
\begin{equation*}
\begin{split}
\epsilon_2\big|t+\tau
&+
\tfrac12\scalar*{\mat x,\xi}-\tfrac12\scalar*{\mat (x^\perp+\xi^\perp),x^\pp+\xi^\pp}
\big|^{1/2}
\le 
\epsilon_2\abs*{t-\tfrac12\scalar*{\mat x^\perp,x^\pp}}^{1/2}
\\
&\quad+
\epsilon_2\abs*{\tau-\tfrac12\scalar*{\mat \xi^\perp,\xi^\pp}}^{1/2}
+
\tfrac{\epsilon_2}2
\abs*{\scalar*{\mat x,\xi}-\scalar*{\mat x^\perp,\xi^\pp}-\scalar*{\mat \xi^\perp,x^\pp}}^{1/2}
\\
&\le 
\alpha x_1
+
\beta\xi_1
+\tfrac{\epsilon_2}2\sqrt{\matn\left(\alpha\beta+2\beta\right)}\,
x_1^{1/2}
\xi_1^{1/2},
\end{split}
\end{equation*}
immediately implying that $q\in C^+(0,\gamma)$.
The case of negative cones is similar.
\end{proof}

\subsection{Intrinsic Lipschitz graphs and functions}

The following definition rephrases~\cite{M14}*{Def.~4.6} and~\cite{FS16}*{Def.~11 and Prop.~3.3}.

\begin{definition}[Intrinsic Lipschitz graph and function]
\label{def:ilip}
The \emph{intrinsic graph} of $\varphi\colon A\to\R$ over the non-empty set $A\subset\W$ is 
\begin{equation*}
\graph(\varphi;A)
=
\set*{\Phi(w):w\in A}
=
\set*{w\star\varphi(w) : w\in A}
\subset\G,
\end{equation*} 
where $\Phi\colon A\to\G$, $\Phi(w)=w\star\varphi(w)$ for $w\in A$, is the \emph{graph map}.
We say that $\varphi$ is \emph{intrinsic Lipschitz} on~$A$ with \emph{intrinsic Lipschitz constant} $L\in[0,+\infty)$, and we write $\varphi\in\ilip(A)$ and $L=\ilip(\varphi;A)$, 
if, for $L>0$, 
\begin{equation*}
\graph(\varphi;A)\cap C(p,1/L)=\emptyset
\quad
\text{for all}\ p\in\graph(\varphi;A),
\end{equation*}
and $\varphi$ constant on $A$ for $L=0$. 
Equivalently, for all $p,q\in\graph(\varphi;A)$, it holds that 
\begin{equation*}
|\varphi(\pi_\W(p))-\varphi(\pi_\W(q))|
\le 
L\|\pi_\W(q^{-1}\star p)\|_\infty.
\end{equation*}
We use the shorthand $\ilip(\varphi)=\ilip(\varphi;\W)$.
\end{definition}

As established in~\cite{FS16}*{Prop.~3.8}, intrinsic Lipschitz functions are continuous---in fact, $\tfrac12$-H\"older continuous, since $\G$ is a Carnot group of step~$2$.

The following result, which generalizes~\cite{M14}*{Prop.~4.8}, is a particular instance of~\cite{FS16}*{Th.~4.1} and 
\cite{V22}*{Th.~1.5}.
The key point here is to provide an explicit bound on the intrinsic Lipschitz constant of the intrinsic Lipschitz extension.

\begin{theorem}[Intrinsic Lipschitz extension]
\label{res:extension}
There is $c=c(\epsilon_2,\matn)>0$ with the following property.
If $\varphi\in\ilip(A)$ for some $\emptyset\ne A\subset \W$, with $L=\ilip(\varphi;A)$, then there exists $\psi\in\ilip(\W)$ such that $\psi(w)=\varphi(w)$ for all $w\in A$, $\|\psi\|_{L^\infty(\W)}=\|\varphi\|_{L^\infty(A)}$ and
\begin{equation*}
\ilip(\psi)\le c\,\max\set*{L,L^4}.
\end{equation*}
Here $\matn>0$ is the constant in~\eqref{eq:b_constold}.
\end{theorem}

\begin{proof}
Assume $L>0$ to avoid trivialities, let $\alpha=1/L$, and define the open set
\begin{equation*}
E=\bigcup_{w\in A} C^+(\Phi(w),\alpha)\ne\emptyset.
\end{equation*}
Setting 
$\beta=\frac{\alpha^2}{\alpha+2}\,\frac{4}{4+\matn\epsilon_2^2}$,
by \cref{res:cones}\ref{item:cone_abc} we get that, if $q\in E$, then $C^+(q,\beta)\subset E$.
By an elementary continuity argument, the latter inclusion also holds for any $q\in\partial E$, the topological boundary of $E$.
Consequently, if $p,q\in\partial E$, then $p\notin C^+(q,\beta)$.
As in the proof of~\cite{M14}*{Prop.~4.8}, we thus get that $\psi\colon \W\to\R$, given by\begin{equation*}
\psi(w)=s_w\mathrm e_1,
\quad
\text{where}
\
s_w=\min\set*{\inf\set*{s\in\R:w\star s\mathrm e_1\in E},\|\varphi\|_{L^\infty(A)}}
\
\text{for}\
w\in\W,
\end{equation*}
is well defined and such that $\psi(w)=\varphi(w)$ for all $w\in A$, $\graph(\psi;\W)\subset\partial E$ and $\|\psi\|_{L^\infty(\W)}=\|\varphi\|_{L^\infty(A)}$.
Finally, given $p,q\in\graph(\psi;\W)$, arguing as in the proof of~\cite{M14}*{Prop.~4.8} and in virtue of \cref{res:cones}\ref{item:cone_inv}, we get that, if $p\notin C^+(q,\beta)$, then $q\notin C^-(p,\gamma)$, where $\gamma>0$ is chosen such that $\beta=\gamma+\epsilon_2\sqrt{\gamma\matn}$, that is, $\gamma=\frac14\left(\sqrt{\epsilon_2^2\matn+4\beta}-\epsilon_2\sqrt{\matn}\right)^2$.
In particular, $\psi\in\ilip(\W)$ with $\ilip(\psi)=1/\gamma$, and a simple computation yields that $\ilip(\psi)\le c\max\set*{L,L^4}$ with $c=c(\epsilon_2,\matn)$>0, concluding the proof.
\end{proof}

\subsection{Intrinsic gradient}

The following definition rephrases~\cite{ADDL20}*{Def.~3.1}. 

\begin{definition}[$\varphi$-gradient]
Let $A\subset\W$ be a non-empty open set and $\varphi\in C(A)$.
The \emph{$\varphi$-gradient of $f\in C^\infty(\W)$} is
$\nabla^\varphi f
=
(\nabla^\varphi_1f,\dots \nabla^\varphi_{m_1-1} f)\colon A\to\R^{m_1-1}$,
where
\begin{equation*}
\nabla^\varphi_i f(w)
=
X_{i+1}(f\circ\pi_\W)(\Phi(w))
\end{equation*}
for all $w\in A$ and each $i=1,\dots,m_1-1$.
\end{definition}

We can hence give the following definition, see the first lines of the proof of~\cite{ADDL20}*{Prop.~4.10} and~\cite{D20}*{Def.~3.2}.

\begin{definition}[Intrinsic gradient]
Let $A\subset\W$ be a non-empty open set.
The \emph{intrinsic gradient} of $\varphi\in C(A)$ is the distribution $\nabla^\varphi\varphi=(\nabla^\varphi\varphi_1,\dots,\nabla^\varphi\varphi_{m_1})$ acting as
\begin{equation*}
\scalar*{\nabla_i^\varphi\varphi,\vartheta}
=
\int_A\varphi\,(\nabla_i^{\varphi})^*\vartheta\di\leb^{n-1}
\quad
\text{for any}\ \vartheta\in C^1_c(A),
\end{equation*}   
where $(\nabla_i^{\varphi})^*$ is the formal adjoint of $\nabla^\varphi_i$, for each $i=1,\dots,m_1$. 
\end{definition}

The following result, which is an immediate consequence of~By \cite{D20}*{Prop.~5.3}, generalizes~\cite{CMPS14}*{Prop.~4.4} to any Carnot group of step~$2$.

\begin{theorem}[Bound on the intrinsic gradient]
\label{res:bounded_igrad}
Let $A\subset\W$ be a non-empty open set.
If $\varphi\in\ilip(A)$, then $\nabla^\varphi\varphi\in L^\infty(A;\R^{m_1-1})$, with $\|\nabla^\varphi\varphi\|_{L^\infty(A)}\le C_L$, for some $C_L>0$ depending on $L=\ilip(\varphi;A)$ only.
\end{theorem}

\subsection{Intrinsic area formula}
The following result follows from~\cite{D20}*{Lem.~5.2 and Th.~5.7} (also see~\cite{ADDL20}*{Prop.~4.10(d)} for more regular functions).

\begin{theorem}[Intrinsic area formula]
\label{res:area}
Let $A\subset\W$ be an non-empty open set.
The \emph{intrinsic epigraph of $\varphi\in\lip_\W(A)$ over $A$},
\begin{equation*}
E_{\varphi,A}
=
\set*{\exp(sX_1): w\in A,\ s>\varphi(w)}
\subset\G,
\end{equation*}
has locally finite $\G$-perimeter in $A\star\R$, its inner horizontal normal is given by
\begin{equation*}
\nu_{E_{\varphi,A}}
(w\star\varphi(w))
=
\left(
\frac1{\sqrt{1+|\nabla^\varphi\varphi(w)|^2}}
,
\frac{-\nabla^\varphi\varphi(w)}{\sqrt{1+|\nabla^\varphi\varphi(w)|^2}}
\right)
\quad
\text{for $\leb^{n-1}$-a.e.}\ w\in A,
\end{equation*}
and its $\G$-perimeter satisfies the \emph{intrinsic area formula} 
\begin{equation}
\label{eq:area}
P(E_{\varphi,A};A'\star\R)
=
\int_{A'}\sqrt{1+|\nabla^\varphi\varphi(w)|^2}\di\leb^{n-1}(w)
\quad
\text{for any}\ A'\subset A.
\end{equation}
\end{theorem}

It is worth noticing that, via well-known standard arguments, the area formula~\eqref{eq:area} can be generalized as 
\begin{equation*}
\int_{\partial E_{\varphi,A}\cap A'\star\R} g(p)\di\mu_E(p)
=
\int_{A'}g(\Phi(w))\sqrt{1+|\nabla^\varphi\varphi(w)|^2}\di\leb^{n-1}(w)
\end{equation*}
whenever $g\colon\partial E_{\varphi,A}\to\R$ is a Borel function.

\section{Intrinsic Lipschitz approximation}\label{ilasec}

Throughout this section, we assume that $(\G,\star)$ is a plentiful group as in \cref{def:plentiful}.
Our approach adapts some ideas of~\cites{M14,MV15,MS17} to the present more general setting.

\subsection{Small-excess position}

The following result corresponds to~\cite{MV15}*{Lem.~3.3}, which was stated in the setting of the Heisenberg groups $\mathbb H^n$, $n\ge 2$ (also see~\cite{M12}*{Lem.~22.10} for the Euclidean case).
The very same result holds for any plentiful group, with identical proof, thanks to \cref{res:flat}.

\begin{lemma}[Small-excess position]
\label{res:small_exc_pos}
For any $s\in(0,1)$, $\Lambda\in[0,+\infty)$ and $r\in(0,+\infty]$ with $\Lambda r_0\le1$, there exists $\omega(s,\Lambda,r_0)>0$ with the following property.
If $E\subset\G$ is a $(\Lambda,r_0)$-minimizer of the $\G$-perimeter in $C_2$, with $0\in\partial E$ and $\exc(2)\le\omega(s,\Lambda,r_0)$, then 
\begin{equation*}
\sup\set*{|\hgt(p)| : p\in C_1\cap\partial E}
\le 
s,
\end{equation*} 
\begin{equation*}
\leb^{n-1}\left(\set*{p\in E\cap C_1 : \hgt(p)>s}\right)=0,
\end{equation*}
\begin{equation*}
\leb^{n-1}\left(\set*{p\in  C_1\setminus E : \hgt(p)<-s}\right)=0.
\end{equation*}
\end{lemma}

\subsection{Intrinsic Lipschitz approximation}

We are now finally ready to state and prove our main result, which generalizes~\cite{M14}*{Th.~5.1} and---only partially---\cite{MS17}*{Th.~3.1} to the setting of plentiful groups.
Its proof revisits that of~\cite{MS17}*{Th.~3.1}, closely following the usual approach in the Euclidean setting, see~\cite{M12}*{Th.~23.7}.

\begin{theorem}[Intrinsic Lipschitz approximation]
\label{res:lipa}
For any $L\in(0,1)$, $\Lambda\in[0,+\infty)$ and $r_0\in(0,+\infty]$, with $\Lambda r_0\le 1$, there exist $\varepsilon,C>0$, depending on $L$, $\Lambda$ and $r_0$ only, with the following property.
If $E\subset\G$ is a $(\Lambda,r_0)$-minimizer of the $\G$-perimeter in $C_{324}$ with $\exc(324)\le\varepsilon$ and $0\in\partial E$, then, letting 
\begin{equation*}
M=C_1\cap\partial E,
\quad
M_0=\set*{
q\in M : \sup_{0<r<16}\exc(q,r)\le\varepsilon},
\end{equation*}
there exists an intrinsic Lipschitz funciton $\varphi\colon\W\to\R$ such that 
\begin{equation}
\label{eq:lipa_norms}
\sup_\W|\varphi|\le L,
\quad
\ilip(\varphi)\le c(\epsilon_2,\matn)\,L,
\end{equation}
\begin{equation}
\label{eq:lipa_graph}
M_0\subset M\cap\Gamma,
\quad
\Gamma=\graph(\varphi;D_1),
\end{equation}
\begin{equation}
\label{eq:lipa_diffsim}
\shaus_\infty^{Q-1}(M\bigtriangleup\Gamma)
\le 
C\,\exc(324),
\end{equation}
\begin{equation}
\label{eq:lipa_energy}
\int_{D_1}|\nabla^\varphi\varphi|^2\di\leb^{n-1}
\le 
C\,\exc(324),
\end{equation}
where $c(\epsilon_2,\matn)>0$ is the constant given by \cref{res:extension}.
\end{theorem}

\begin{proof}
Let $L\in(0,1)$, $\Lambda\in[0,+\infty)$ and $r_0\in(0,+\infty]$ be fixed and let $E$, $M$ and $M_0$ be as in the statement.
With the notation of \cref{res:small_exc_pos}, we choose
\begin{equation}
\label{eq:eee}
\varepsilon
=
\min\set*{
\frac{\omega(L,\Lambda,r_0)}{162^{Q-1}},\omega\left(L,8\Lambda,\frac{r_0}8\right)
}.
\end{equation}
The proof is then divided into three steps.

\vspace*{1ex}

\textit{Step~1: construction of $\varphi$}.
Since $\exc(324)\le\omega(L,\Lambda,r_0)$ by~\eqref{eq:eee}, by \cref{res:small_exc_pos} we have
\begin{equation}
\label{eq:protobounded}
\sup\set*{|\hgt(p)| : p\in C_1\cap\partial E}
\le 
L.
\end{equation}
Given $p\in M$ and $q\in M_0$, we have $p,q\in C_1$, so that $\lambda=d_C(p,q)<8$ by~\eqref{eq:inclusions_BC}. 
By \cref{rem:scaling}, the set $F=\delta_{\lambda^{-1}}(q^{-1}\star E)$ is a $(\lambda\Lambda,\frac{r_0}{\lambda})$-minimizer of the $\G$-perimeter in $C_{\frac{324}\lambda}(q^{-1})$ with $0\in\partial F$.
Since $C_{\frac{324}\lambda}(q^{-1})\supset C_{\frac{81}2}(q^{-1})\supset C_2$ for all $q\in C_1$, by the invariance properties of the excess and by definition of $M_0$, we infer that
\begin{equation}
\label{eq:F_exc}
\exc(F,0,2)
=
\exc(E,q,2\lambda)
\le 
\varepsilon.
\end{equation}
Recalling that $\lambda<8$, $F$ is a $(8\Lambda,\frac{r_0}{8})$-minimizer of the $\G$-perimeter in $C_{\frac{324}\lambda}(q^{-1})$. 
Since $\varepsilon\le\omega(L,8\Lambda,\frac{r_0}8)$ due to~\eqref{eq:eee}, by~\eqref{eq:F_exc} and again by \cref{res:small_exc_pos}, we infer that
\begin{equation*}
\sup\set*{|\hgt(v)| : v\in C_1\cap\partial F}
\le 
L.
\end{equation*}
In particular, choosing 
$v=\delta_{\lambda^{-1}}(q^{-1}\star p)\in C_1\cap\partial F$, 
we get that
\begin{equation*}
|\hgt(q^{-1}\star p)|
\le 
L\,d_C(p,q).
\end{equation*}
Since $L<1$, the above inequality, combined with the definition in~\eqref{eq:norm_C}, yields that $d_C(p,q)=\|\pi_\W(q^{-1}\star p)\|_\infty$, so that 
\begin{equation}
\label{eq:protolip}
|\hgt(q^{-1}\star p)|
\le 
L\,\|\pi_\W(q^{-1}\star p)\|_\infty
\quad
\text{for all}\ p\in M,\ q\in M_0.
\end{equation}
As a consequence, the projection $\pi_\W$ is invertible on $M_0$ and we can thus define a function $\varphi\colon\pi_\W(M_0)\to\R$ by letting $\varphi(\pi_\W(p))=\hgt(p)$ for all $p\in M_0$.
Due to~\eqref{eq:protolip}, we get that 
\begin{equation*}
|\varphi(\pi_\W(p))-\varphi(\pi_\W(q))|
\le 
L\,\|\pi_\W(q^{-1}\star p)\|_\infty
\quad
\text{for all}\ p,q\in M_0,
\end{equation*}
so that $\varphi\in\ilip(\pi_\W(M_0))$ with $\ilip(\varphi;\pi_\W(M_0))\le L<1$, in virtue of \cref{def:ilip}.
Since $M_0\subset M$, from~\eqref{eq:protobounded} we also get that 
$|\varphi(\pi_\W(p))|\le L$ for all $p\in M_0$.
By \cref{res:extension}, we can find an extension of $\varphi$ to the whole $\W$ (for which we keep the same notation) such that $\ilip(\varphi)\le c(\epsilon_2,\matn)\,L$ and $|\varphi(w)|\le L$ for all $w\in\W$. 
By construction, we also get that $M_0\subset M\cap\Gamma$, where $\Gamma=\graph(\varphi;D_1)$.
This proves~\eqref{eq:lipa_norms} and~\eqref{eq:lipa_graph}.

\vspace{1ex}

\textit{Step~2: covering argument}.
We now prove~\eqref{eq:lipa_diffsim} via a covering argument.
By definition of $M_0$, for each $q\in M\setminus M_0$ there exists $r_q\in(0,16)$ such that 
\begin{equation}
\label{eq:exc_sotto}
\int_{C_{r_q}(q)\cap\partial E}
\frac{|\nu_E-\nu|^2}2
\di\shaus_\infty^{Q-1}
>
\varepsilon\, r_q^{Q-1}.
\end{equation}
%where $\nu=-X_1$ as in \cref{def:excess}. 
The family of balls 
$\set*{B_{2r_q}(q) : q\in M\setminus M_0}$
is a covering of $M\setminus M_0$. 
By Vitali's Covering Lemma, there exist $q_h\in M\setminus M_0$, for $h\in\N$, such that the countable subfamily 
$\set*{B_{2r_h}(q_h) : r_h={r_{q_h}},\ q_h\in M\setminus M_0,\ h\in\N}$
is disjoint, and the family $\set*{B_{10r_h}(q_h) :h\in\N}$ is still a covering of $M\setminus M_0$.
Therefore, by \cref{res:density}, we can estimate
\begin{equation}
\label{eq:stima_M-M0}
\begin{split}
\shaus_\infty^{Q-1}(M\setminus M_0)
&\le 
\sum_{h\in\N}
\shaus_\infty^{Q-1}\big((M\setminus M_0)\cap B_{10r_h}(q_h)\big)
\\
&\le 
\sum_{h\in\N}
\shaus_\infty^{Q-1}(M\cap B_{10r_h}(q_h))
\le 
c 
\sum_{h\in\N}
r_h^{Q-1},
\end{split}
\end{equation}
where $c>0$ is a constant that does not dependent on $L$, $\Lambda$ or $r_0$. 
Now note that $B_{10r_h}(q_h)\subset C_{324}$ for all $h\in\N$, since, in virtue of~\eqref{eq:norms_inftyC}, any $p\in B_{10r_h}(q_h)$ satisfies  
\begin{equation*}
\|p\|_C
\le
2\|p\|_\infty
\le 
2d_\infty(p,q_h)
+
2\|q_h\|_\infty
<
20r_h+4\|q_h\|_C
<
324.
\end{equation*}
Moreover, since $C_{r_h}(q_h)\subset B_{2r_h}(q_h)$ by~\eqref{eq:inclusions_BC}, also the cylinders $\set*{C_{r_h}(q_h) : h\in\N}$ are disjoint and contained in $C_{324}$.
Therefore, by combining~\eqref{eq:exc_sotto} with~\eqref{eq:stima_M-M0}, we get that 
\begin{equation*}
\shaus_\infty^{Q-1}(M\setminus M_0)
\le 
\frac c\varepsilon
\sum_{h\in\N}
\int_{C_{r_h}(q_h)\cap\partial E}
\frac{|\nu_E-\nu|^2}2
\di\shaus_\infty^{Q-1}
\le 
\frac c\varepsilon\,
\exc(324).
\end{equation*}
Consequently, since $M\setminus\Gamma\subset M\setminus M_0$, we conclude that 
\begin{equation*}
\shaus_\infty^{Q-1}(M\setminus\Gamma)
\le 
\frac c\varepsilon\,
\exc(k),
\end{equation*}
which is the first half of~\eqref{eq:lipa_diffsim}.
To prove the second half of~\eqref{eq:lipa_diffsim}, we observe that 
\begin{equation*}
\exc(2)
\le 
\left(\frac{324}2\right)^{Q-1}\exc(324)
\le 
\omega(L,\Lambda,r_0),
\end{equation*}
thanks to the properties of the excess and~\eqref{eq:eee}.
Hence, by~\eqref{eq:small_height_est} in \cref{res:exc_meas},
\begin{align*}
\shaus_\infty^{Q-1}(\Gamma\setminus M)
&=
\int_{\pi_\W(\Gamma\setminus M)}
\sqrt{1+|\nabla^\varphi\varphi|^2}\di\leb^{n-1}
\\
&\le
\sqrt{1+\|\nabla^\varphi\varphi\|_{L^\infty(\W)}^2}
\,\leb^{n-1}(\pi_\W(\Gamma\setminus M))
\\
&\le
\sqrt{1+\|\nabla^\varphi\varphi\|_{L^\infty(\W)}^2}
\,\shaus_\infty^{Q-1}
\big(
M\cap \pi_\W^{-1}(\pi_\W(\Gamma\setminus M))
\big).
\end{align*}
In virtue of \cref{res:bounded_igrad}, we can estimate 
\begin{equation*}
\sqrt{1+\|\nabla^\varphi\varphi\|_{L^\infty(\W)}^2}
\le C_L,
\end{equation*}
where $C_L>0$ depends on $L$ only.
Since 
$M\cap \pi_\W^{-1}(\pi_\W(\Gamma\setminus M))
\subset M\setminus\Gamma$,
we get that 
\begin{equation*}
\shaus_\infty^{Q-1}(\Gamma\setminus M)
\le 
C_L\,
\shaus_\infty^{Q-1}(M\setminus\Gamma)
\le
\frac{C_L}\varepsilon\,
\exc(k) ,
\end{equation*}
completing the proof of~\eqref{eq:lipa_diffsim}.

\vspace{1ex}

\textit{Step~3: estimate on the $L^2$ energy}.
Finally, we prove~\eqref{eq:lipa_energy}.
By \cref{res:area} and~\cite{AS10}*{Cor.~2.6}, for $\shaus^{Q-1}_\infty$-a.e.\ $p\in M\cap\Gamma$ there exists $\sigma(p)\in\set*{-1,1}$ such that 
\begin{equation*}
\nu_E(p)
=
\sigma(p)\,
\frac{\big(1,-\nabla^\varphi\varphi(\pi_\W(p))\big)}{\sqrt{1+|\nabla^\varphi\varphi(\pi_\W(p))|^2}}.
\end{equation*}
Taking into account that, for $\shaus^{Q-1}_\infty$-a.e.\ $p\in M\cap\Gamma$,
\begin{equation*}
\frac{|\nu_E(p)-\nu(p)|^2}2
=
1-\scalar*{\nu_E(p),\nu(p)}
\ge 
\frac{1-\scalar*{\nu_E(p),\nu(p)}^2}{2},
\end{equation*}
we get that 
\begin{align*}
\exc(1)
&\ge 
\int_{M\cap\Gamma}
\frac{1-\scalar*{\nu_E(p),\nu(p)}^2}{2}
\di\mu_E(p)
=
\frac12
\int_{M\cap\Gamma}
\frac{|\nabla^\varphi\varphi(\pi_\W(p))|^2}{1+|\nabla^\varphi\varphi(\pi_\W(p))|^2}
\di\mu_E(p)
\\
&=
\frac12
\int_{\pi_\W(M\cap\Gamma)}
\frac{|\nabla^\varphi\varphi(w)|^2}{1+|\nabla^\varphi\varphi(w)|^2}
\di\leb^{Q-1}(w).
\end{align*}
By \cref{res:bounded_igrad} and the scaling property of the excess, we get that
\begin{equation*}
\int_{\pi_\W(M\cap\Gamma)}
|\nabla^\varphi\varphi|^2
\di\leb^{Q-1}
\le 
C_L\,\exc(324),
\end{equation*}
where $C_L>0$ depends on $L$ only.
Moreover, by \cref{res:area}, we can estimate
\begin{align*}
\int_{\pi_\W(M\bigtriangleup\Gamma)}
|\nabla^\varphi\varphi|^2
\di\leb^{Q-1}
\le 
\int_{M\bigtriangleup\Gamma}
\frac{|\nabla^\varphi\varphi(\pi_\W(p))|^2}{1+|\nabla^\varphi\varphi(\pi_\W(p))|^2}
\di\mu_E(p)
\le 
\shaus^{Q-1}_\infty(M\bigtriangleup\Gamma),
\end{align*}
and~\eqref{eq:lipa_energy}  immediately follows from~\eqref{eq:lipa_diffsim}.
The proof is complete.
\end{proof}
%\begin{remark}
%As remarked in the Introduction, Theorem \ref{res:lipa} is the analogous of ~\cite{M14}*{Th.~5.1} for plentiful groups. In \cite{MS17}, following the scheme outlined in \cite{M12}*{Section 23.3}, the authors improved \cite{M14} providing the natural reformulation in $\mathbb{H}^n$ of the classical Lipschitz approximation in $\mathbb{R}^n$. The reader might be puzzled about the dependency of the constants $\varepsilon$ and $C$ on $L,\Lambda$ and $r_0$ in Theorem \ref{res:lipa} which is not present in \cite{MS17}. This fact is essentially due to the lack of a height estimate in plentiful groups which is one of the main tool used in \cite{MS17}. The validity of such an estimate is the object of a future work.
%\end{remark}


\begin{bibdiv}
\begin{biblist}

% copy entries from mathscinet for articles/books 

% default entry for arxiv preprint:
%
%\bib{XXX}{article}{
%   author={Family Name, Name},
%   title={insert title},
%   date={insert year},
%   note={Preprint, available at \href{https://arxiv.org/abs/1234.12345}{arXiv:1234.12345}}
%}

\bib{AS10}{article}{
   author={Ambrosio, Luigi},
   author={Scienza, Matteo},
   title={Locality of the perimeter in Carnot groups and chain rule},
   journal={Ann. Mat. Pura Appl. (4)},
   volume={189},
   date={2010},
   number={4},
   pages={661--678},
%   issn={0373-3114},
   review={\MR{2678937}},
   doi={10.1007/s10231-010-0130-9},
}

\bib{ADDL20}{article}{
   author={Antonelli, Gioacchino},
   author={Di Donato, Daniela},
   author={Don, Sebastiano},
   author={Le Donne, Enrico},
   title={Characterizations of uniformly differentiable co-horizontal intrinsic graphs in Carnot groups},
   date={to appear},
   journal={Ann. Inst. Fourier (Grenoble)},
   doi={10.48550/arXiv.2005.11390},
}

\bib{BLU07}{book}{
   author={Bonfiglioli, A.},
   author={Lanconelli, E.},
   author={Uguzzoni, F.},
   title={Stratified Lie groups and potential theory for their
   sub-Laplacians},
   series={Springer Monographs in Mathematics},
   publisher={Springer, Berlin},
   date={2007},
%   pages={xxvi+800},
%   isbn={978-3-540-71896-3},
%   isbn={3-540-71896-6},
   review={\MR{2363343}},
}

\bib{CCM09}{article}{
   author={Calin, Ovidiu},
   author={Chang, Der-Chen},
   author={Markina, Irina},
   title={Geometric analysis on $H$-type groups related to division algebras},
   journal={Math. Nachr.},
   volume={282},
   date={2009},
   number={1},
   pages={44--68},
%   issn={0025-584X},
   review={\MR{2473130}},
   doi={10.1002/mana.200710721},
}

\bib{CapCM09}{article}{
   author={Capogna, Luca},
   author={Citti, Giovanna},
   author={Manfredini, Maria},
   title={Regularity of non-characteristic minimal graphs in the Heisenberg group $\mathbb{H}^1$},
   journal={Indiana Univ. Math. J.},
   volume={58},
   date={2009},
   number={5},
   pages={2115--2160},
%   issn={0022-2518},
   review={\MR{2583494}},
   doi={10.1512/iumj.2009.58.3673},
}

\bib{CCM10}{article}{
   author={Capogna, Luca},
   author={Citti, Giovanna},
   author={Manfredini, Maria},
   title={Smoothness of Lipschitz minimal intrinsic graphs in Heisenberg
   groups $\mathbb{H}^n$, $n>1$},
   journal={J. Reine Angew. Math.},
   volume={648},
   date={2010},
   pages={75--110},
%   issn={0075-4102},
   review={\MR{2774306}},
   doi={10.1515/CRELLE.2010.080},
}

\bib{CHY09}{article}{
   author={Cheng, Jih-Hsin},
   author={Hwang, Jenn-Fang},
   author={Yang, Paul},
   title={Regularity of $C^1$ smooth surfaces with prescribed $p$-mean
   curvature in the Heisenberg group},
   journal={Math. Ann.},
   volume={344},
   date={2009},
   number={1},
   pages={1--35},
%   issn={0025-5831},
   review={\MR{2481053}},
   doi={10.1007/s00208-008-0294-4},
}

\bib{CMPS14}{article}{
   author={Citti, Giovanna},
   author={Manfredini, Maria},
   author={Pinamonti, Andrea},
   author={Serra Cassano, Francesco},
   title={Smooth approximation for intrinsic Lipschitz functions in the
   Heisenberg group},
   journal={Calc. Var. Partial Differential Equations},
   volume={49},
   date={2014},
   number={3-4},
   pages={1279--1308},
%   issn={0944-2669},
   review={\MR{3168633}},
   doi={10.1007/s00526-013-0622-8},
}

\bib{CDKR91}{article}{
   author={Cowling, Michael},
   author={Dooley, Anthony H.},
   author={Kor\'{a}nyi, Adam},
   author={Ricci, Fulvio},
   title={$H$-type groups and Iwasawa decompositions},
   journal={Adv. Math.},
   volume={87},
   date={1991},
   number={1},
   pages={1--41},
   issn={0001-8708},
   review={\MR{1102963}},
   doi={10.1016/0001-8708(91)90060-K},
}

\bib{DGN10}{article}{
   author={Danielli, D.},
   author={Garofalo, N.},
   author={Nhieu, D. M.},
   title={Sub-Riemannian calculus and monotonicity of the perimeter for
   graphical strips},
   journal={Math. Z.},
   volume={265},
   date={2010},
   number={3},
   pages={617--637},
%   issn={0025-5874},
   review={\MR{2644313}},
   doi={10.1007/s00209-009-0533-8},
}

\bib{D20}{article}{
   author={Di Donato, Daniela},
   title={Intrinsic Lipschitz graphs in Carnot groups of step 2},
   journal={Ann. Acad. Sci. Fenn. Math.},
   volume={45},
   date={2020},
   number={2},
   pages={1013--1063},
%   issn={1239-629X},
   review={\MR{4112274}},
   doi={10.5186/aasfm.2020.4556},
}

\bib{FS16}{article}{
   author={Franchi, Bruno},
   author={Serapioni, Raul Paolo},
   title={Intrinsic Lipschitz graphs within Carnot groups},
   journal={J. Geom. Anal.},
   volume={26},
   date={2016},
   number={3},
   pages={1946--1994},
%   issn={1050-6926},
   review={\MR{3511465}},
   doi={10.1007/s12220-015-9615-5},
}

\bib{FSS01}{article}{
   author={Franchi, Bruno},
   author={Serapioni, Raul},
   author={Serra Cassano, Francesco},
   title={Rectifiability and perimeter in the Heisenberg group},
   journal={Math. Ann.},
   volume={321},
   date={2001},
   number={3},
   pages={479--531},
   issn={0025-5831},
   review={\MR{1871966}},
   doi={10.1007/s002080100228},
}

\bib{FSS03}{article}{
   author={Franchi, Bruno},
   author={Serapioni, Raul},
   author={Serra Cassano, Francesco},
   title={On the structure of finite perimeter sets in step 2 Carnot groups},
   journal={J. Geom. Anal.},
   volume={13},
   date={2003},
   number={3},
   pages={421--466},
%   issn={1050-6926},
   review={\MR{1984849}},
   doi={10.1007/BF02922053},
}

\bib{FSS06}{article}{
   author={Franchi, Bruno},
   author={Serapioni, Raul},
   author={Serra Cassano, Francesco},
   title={Intrinsic Lipschitz graphs in Heisenberg groups},
   journal={J. Nonlinear Convex Anal.},
   volume={7},
   date={2006},
   number={3},
   pages={423--441},
%   issn={1345-4773},
   review={\MR{2287539}},
}

\bib{G84}{book}{
   author={Giusti, Enrico},
   title={Minimal surfaces and functions of bounded variation},
   series={Monographs in Mathematics},
   volume={80},
   publisher={Birkh\"{a}user Verlag, Basel},
   date={1984},
%   pages={xii+240},
%   isbn={0-8176-3153-4},
   review={\MR{0775682}},
   doi={10.1007/978-1-4684-9486-0},
}

\bib{K80}{article}{
   author={Kaplan, Aroldo},
   title={Fundamental solutions for a class of hypoelliptic PDE generated by composition of quadratic forms},
   journal={Trans. Amer. Math. Soc.},
   volume={258},
   date={1980},
   number={1},
   pages={147--153},
%   issn={0002-9947},
   review={\MR{0554324}},
   doi={10.2307/1998286},
}

\bib{KS04}{article}{
   author={Kirchheim, Bernd},
   author={Serra Cassano, Francesco},
   title={Rectifiability and parameterization of intrinsic regular surfaces
   in the Heisenberg group},
   journal={Ann. Sc. Norm. Super. Pisa Cl. Sci. (5)},
   volume={3},
   date={2004},
   number={4},
   pages={871--896},
%   issn={0391-173X},
   review={\MR{2124590}},
}

\bib{LT22}{article}{
   author={Le Donne, Enrico},
   author={Tripaldi, Francesca},
   title={A cornucopia of Carnot groups in low dimensions},
   journal={Anal. Geom. Metr. Spaces},
   volume={10},
   date={2022},
   number={1},
   pages={155--289},
   review={\MR{4490195}},
   doi={10.1515/agms-2022-0138},
}

\bib{M12}{book}{
   author={Maggi, Francesco},
   title={Sets of finite perimeter and geometric variational problems},
   series={Cambridge Studies in Advanced Mathematics},
   volume={135},
%   note={An introduction to geometric measure theory},
   publisher={Cambridge University Press, Cambridge},
   date={2012},
%   pages={xx+454},
%   isbn={978-1-107-02103-7},
   review={\MR{2976521}},
   doi={10.1017/CBO9781139108133},
}

\bib{M03}{article}{
   author={Magnani, Valentino},
   title={A blow-up theorem for regular hypersurfaces on nilpotent groups},
   journal={Manuscripta Math.},
   volume={110},
   date={2003},
   number={1},
   pages={55--76},
   issn={0025-2611},
   review={\MR{1951800}},
   doi={10.1007/s00229-002-0303-y},
}

\bib{M17}{article}{
   author={Magnani, Valentino},
   title={A new differentiation, shape of the unit ball, and perimeter measure},
   journal={Indiana Univ. Math. J.},
   volume={66},
   date={2017},
   number={1},
   pages={183--204},
%   issn={0022-2518},
   review={\MR{3623407}},
   doi={10.1512/iumj.2017.66.6007},
}

\bib{M14}{article}{
   author={Monti, Roberto},
   title={Lipschitz approximation of $\mathbb{H}$-perimeter minimizing boundaries},
   journal={Calc. Var. Partial Differential Equations},
   volume={50},
   date={2014},
   number={1-2},
   pages={171--198},
%   issn={0944-2669},
   review={\MR{3194680}},
   doi={10.1007/s00526-013-0632-6},
}

\bib{M15}{article}{
   author={Monti, Roberto},
   title={Minimal surfaces and harmonic functions in the Heisenberg group},
   journal={Nonlinear Anal.},
   volume={126},
   date={2015},
   pages={378--393},
%   issn={0362-546X},
   review={\MR{3388885}},
   doi={10.1016/j.na.2015.03.013},
}

\bib{MSV08}{article}{
   author={Monti, Roberto},
   author={Serra Cassano, Francesco},
   author={Vittone, Davide},
   title={A negative answer to the Bernstein problem for intrinsic graphs in the Heisenberg group},
   journal={Boll. Unione Mat. Ital. (9)},
   volume={1},
   date={2008},
   number={3},
   pages={709--727},
%   issn={1972-6724},
   review={\MR{2455341}},
}

\bib{MS17}{article}{
   author={Monti, Roberto},
   author={Stefani, Giorgio},
   title={Improved Lipschitz approximation of $H$-perimeter minimizing
   boundaries},
%   language={English, with English and French summaries},
   journal={J. Math. Pures Appl. (9)},
   volume={108},
   date={2017},
   number={3},
   pages={372--398},
%   issn={0021-7824},
   review={\MR{3682744}},
   doi={10.1016/j.matpur.2017.04.002},
}

\bib{MV15}{article}{
   author={Monti, Roberto},
   author={Vittone, Davide},
   title={Height estimate and slicing formulas in the Heisenberg group},
   journal={Anal. PDE},
   volume={8},
   date={2015},
   number={6},
   pages={1421--1454},
%   issn={2157-5045},
   review={\MR{3397002}},
   doi={10.2140/apde.2015.8.1421},
}

\bib{P06}{article}{
   author={Pauls, Scott D.},
   title={$H$-minimal graphs of low regularity in $\mathbb{H}^1$},
   journal={Comment. Math. Helv.},
   volume={81},
   date={2006},
   number={2},
   pages={337--381},
%   issn={0010-2571},
   review={\MR{2225631}},
   doi={10.4171/CMH/55},
}

\bib{R09}{article}{
   author={Ritor\'{e}, Manuel},
   title={Examples of area-minimizing surfaces in the sub-Riemannian Heisenberg group $\mathbb{H}^1$ with low regularity},
   journal={Calc. Var. Partial Differential Equations},
   volume={34},
   date={2009},
   number={2},
   pages={179--192},
   issn={0944-2669},
%   review={\MR{2448649}},
   doi={10.1007/s00526-008-0181-6},
}

\bib{SS82}{article}{
   author={Schoen, Richard},
   author={Simon, Leon},
   title={A new proof of the regularity theorem for rectifiable currents which minimize parametric elliptic functionals},
   journal={Indiana Univ. Math. J.},
   volume={31},
   date={1982},
   number={3},
   pages={415--434},
%   issn={0022-2518},
   review={\MR{0652826}},
   doi={10.1512/iumj.1982.31.31035},
}

\bib{S16}{article}{
   author={Serra Cassano, Francesco},
   title={Some topics of geometric measure theory in Carnot groups},
   conference={
      title={Geometry, analysis and dynamics on sub-Riemannian manifolds. Vol. 1},
   },
   book={
      series={EMS Ser. Lect. Math.},
      publisher={Eur. Math. Soc., Z\"{u}rich},
   },
%   isbn={978-3-03719-162-0},
   date={2016},
   pages={1--121},
   review={\MR{3587666}},
}

\bib{SV14}{article}{
   author={Serra Cassano, Francesco},
   author={Vittone, Davide},
   title={Graphs of bounded variation, existence and local boundedness of
   non-parametric minimal surfaces in Heisenberg groups},
   journal={Adv. Calc. Var.},
   volume={7},
   date={2014},
   number={4},
   pages={409--492},
%   issn={1864-8258},
   review={\MR{3276118}},
   doi={10.1515/acv-2013-0105},
}

\bib{V22}{article}{
   author={Vittone, Davide},
   title={Lipschitz graphs and currents in Heisenberg groups},
   journal={Forum Math. Sigma},
   volume={10},
   date={2022},
   pages={Paper No. e6, 104},
   review={\MR{4377000}},
   doi={10.1017/fms.2021.84},
}

\end{biblist}
\end{bibdiv}
\end{document}